\documentclass[a4paper,10pt]{article}
\usepackage{fullpage}
\usepackage{amsmath,amssymb,latexsym,amsthm}
\usepackage[pdftex]{graphicx}
\usepackage{graphics}
\usepackage{psfrag}
\usepackage{color, epsfig}
\usepackage{mathrsfs}
\usepackage[english]{babel}
\usepackage[utf8]{inputenc}  
\usepackage[T1]{fontenc}  

\def\R{{\mathbb R}}
\def\N{{\mathbb N}}

\def\C{{\mathbb C}}

\def\ds{\displaystyle}

\def\div{\operatorname{div}\,}
\newcommand{\norm}[1]{\left\Vert#1\right\Vert}

\renewcommand{\leq}{\leqslant}
\renewcommand{\geq}{\geqslant}

\newtheorem{theorem}{Theorem}[section]

\newtheorem{proposition}[theorem]{Proposition}
\newtheorem{lemma}[theorem]{Lemma}
\newtheorem{remark}[theorem]{Remark}
\numberwithin{equation}{section}

\title{Switching controls for analytic semigroups and applications to parabolic systems\footnote{This project has been partially supported by the program CAPES/MATH-AMSUD under the project ACIPDE number: 8881.368715/2019-01.  The first author was also partially supported by CNPq-Brazil by the Grant 2019/0014 of Para\'iba State Research Foundation (FAPESQ-PB). 
The second author has been also supported by the Agence Nationale de la Recherche, Project IFSMACS, grant ANR-15-CE40-0010, and by the CIMI Labex, Toulouse, France, under grant ANR-11-LABX-0040-CIMI. The third author was partially supported by CNPq by the grant $313148$/$2017$-1 (Brazil), by Propesq - UFPE edital Qualis A (Brazil) and by grant MTM$2016$-$76990$-P, Ministry of Economy and Competitiveness (Spain).}}

\author{
	Felipe Wallison Chaves Silva\footnote{Department of Mathematics, Federal University of Para\'iba, UFPB, CEP 58050-085, Jo\~ao Pessoa-PB, Brazil.
E-mail: {\tt fchaves@mat.ufpb.br}.}
	\and
	Sylvain Ervedoza\footnote{Institut de Math\'ematiques de Bordeaux UMR 5251, Universit\'e de Bordeaux, Bordeaux INP, CNRS, F-33400 Talence, France. E-mail: {\tt sylvain.ervedoza@math.u-bordeaux.fr}}
	\and 
	Diego Araujo de Souza\footnote{Department of Mathematics, Federal University of Pernambuco, UFPE, CEP 50740-545, Recife, PE, Brazil.
E-mail: {\tt diego.souza@dmat.ufpe.br}}
}

\date\today
\begin{document}
\maketitle
\begin{abstract}
	In this work, we push further the analysis of the problem of switching controls proposed in \cite{ZuaSwitch}. The problem consists in the following one: assuming that one can control a system using two or more actuators, does there exist a control strategy such that at all times, only one actuator is active? We answer positively to this question when the controlled system corresponds to an analytic semigroup spanned by a positive self-adjoint operator which is null-controllable in arbitrary small times. Similarly as in \cite{ZuaSwitch}, our proof relies on analyticity arguments and will also work in finite dimensional setting and under some further spectral assumptions when the operator spans an analytic semigroup but is not necessarily self-adjoint.
\end{abstract}
%
%
%
%
%%%%%%%%%%%%%%%%%%%%%%%%%%%%%%%%%%%%%%%%%%%%%%%%%%%%%%%%%%%%%%%%%%%%%%%%%%%%%%%%%%%%%%%%%%%%%%%%%%%%%%%%%%%%%%%%%%%%%%%%%%%%%%%%%%%%%%%
%
%
%%%%%%%%%%%%%%%%%%%%%%%%%%%%%%%%%%%%%%%%%%%%%%%%%%%%%%%%%
%
%
%
%
\section{Introduction}
%
%
%%%%%%%%%%%%%%%%%%%%%%%%%%%%%%%%%%%%%%%%%%%%%%%%%%%%%%%%%
%

%
\paragraph{Settings and main results.}
In this article, we are interested in the following system:
\begin{equation}
	\label{Main-Eq-u}
	y' + A y = B u, \quad t \in (0,T), \qquad y(0) = y_0 \in H. 
\end{equation}
Here, $y$ is the state variable, assumed to belong to an Hilbert space $H$, $'$ denotes the time derivative and $A$ describes the free dynamics and $-A$ generates a $C^0$ semigroup. The function $u$ is the control, acting on the system through the control operator $B$, which we assumed to be in $\mathscr{L}(U,H)$, where $U$ is an Hilbert space, and $u$ will be searched in the space $L^2(0,T; U)$,  {with $T>0$}. 

Controllability of systems of the form \eqref{Main-Eq-u} have been analyzed thoroughly in many works. We do not intend to give an exhaustive account of the theory, and we simply refer to the textbook \cite{TWBook}.

Here, we focus on the case where $U$ can be identified to $U_1 \times U_2$ through an isomorphism, \emph{i.e.}
\begin{equation}
	\label{Link-U1-2-U}
	\hbox{there exists a linear isomorphism } \pi : U_1 \times U_2 \to U, 
\end{equation}
so that we can associate to $B \in \mathscr{L}(U,H)$ two operators $B_1 \in \mathscr{L}(U_1, H)$ and $B_2 \in \mathscr{L}(U_2, H)$ such that 
\begin{equation}
	\label{Def-B1-B2}
	\forall (u_1, u_2) \in U_1 \times U_2,  \quad B \pi (u_1, u_2) = B_1 u_1 + B_2 u_2.
\end{equation}
The control problem \eqref{Main-Eq-u} can then be rewritten as 
\begin{equation}
	\label{Main-Eq-u-1-2}
	y' + A y = B_1 u_1+ B_2 u_2,  \quad { t \in (0,T)} , \qquad y(0) = y_0, 
\end{equation}
with $u_1 \in L^2(0,T; U_1)$ and $u_2 \in L^2(0,T; U_2)$.

The question we are interested in is the possibility to construct \emph{switching controls}, that is,  controls $u_1 \in L^2(0,T; U_1)$ and $u_2 \in L^2(0,T; U_2)$ such that
\begin{equation}
	\label{Def-Switching-Controls}
	\hbox{a.e. in } t \in (0,T), \quad	\| u_1(t) \|_{U_1} \| u_2(t)\|_{U_2} = 0.  
\end{equation}
Informally, this means that at each time $t$, only one control is active. 

Of course, one cannot expect to have better controllability properties for \eqref{Main-Eq-u-1-2} under the condition \eqref{Def-Switching-Controls} than for the general case \eqref{Main-Eq-u}. We shall thus assume some controllability properties for \eqref{Main-Eq-u} and discuss what can be obtained for the control problem \eqref{Main-Eq-u-1-2} under the condition \eqref{Def-Switching-Controls}.

More precisely, we will assume that the system \eqref{Main-Eq-u} is null-controllable in arbitrary small times, \emph{ i.e.} for all $T>0$, there exists a constant $C_{T}$ such that for all $y_0 \in H$, there exists $u \in L^2(0,T; U)$ such that the solution $y$ of \eqref{Main-Eq-u} satisfies 
\begin{equation}
	\label{Null-Cont-Req}
	y(T) = 0,  
\end{equation}
and 
\begin{equation}
	\label{Bound-Control}
	\norm{u}_{L^2(0,T; U)} \leq C_T \norm{y_0}_H.
\end{equation}

In fact, we will rather use the following equivalent observability property (see \emph{e.g.} \cite[Theorem 11.2.1]{TWBook}): For all $T>0$, there exists $C_T$ such that for all $z_T \in H$, the solution $z$ of 
\begin{equation}
	\label{Adj-Eq}
	- z' + A^* z = 0, \quad t \in (0,T), \qquad z(T) = z_T \in H, 
\end{equation}
satisfies
\begin{equation}
	\label{Obs-T=0}
	\| z(0) \|_H \leq C_T \| B^* z \|_{L^2(0,T; U)}.
\end{equation}

Our goal then is to show the following result: 
\begin{theorem}
	\label{Thm-Main}
	Let us assume one of the two conditions: 
	\begin{itemize}
		\item $A: \mathscr{D}(A) \subset H \to H$ is a self-adjoint positive definite operator with compact resolvent, $H$ being a Hilbert space; 
		\item $H$ is a finite dimensional vector space.
	\end{itemize}
	Let $B \in \mathscr{L}(U,H)$, where $U$ is an Hilbert space, and assume that $U$ is isomorphic to $U_1 \times U_2$ for some Hilbert spaces $U_1$ and $U_2$, and define $B_1$ and $B_2$ as in \eqref{Def-B1-B2}.

	We also assume that system \eqref{Main-Eq-u} is null-controllable in arbitrary small times. 

	Then system \eqref{Main-Eq-u-1-2} is null-controllable in arbitrary small times with switching controls, \emph{i.e.} satisfying \eqref{Def-Switching-Controls}. To be more precise, given any $T>0$ and any $y_0 \in H$, there exist control functions $u_1 \in L^2(0,T; U_1)$ and $u_2 \in L^2(0,T; U_2)$ such that the solution $y$ of \eqref{Main-Eq-u-1-2} satisfies \eqref{Null-Cont-Req} while the control functions satisfy the switching condition \eqref{Def-Switching-Controls}.
\end{theorem}

The proof of Theorem \ref{Thm-Main} is given in Section \ref{Sec-Proof-Main}. It is strongly inspired by the work \cite{ZuaSwitch} and revisits two ideas which are already presented there but that we exploit further. Indeed, to construct the controls $u_1$ and $u_2$, 
 we minimize, for $z_T \in H$, the functional 
\begin{equation}
	\label{Def-J-H}
	J(z_T) = \frac{1}{2} \int_0^T \max \{ \| B_1^* z(t)\|_{U_1}^2, \alpha(t) \| B_2^* z(t)\|_{U_2}^2 \} \, dt + \langle y_0, z(0) \rangle_H, 
\end{equation}
where $z$ is the solution of the adjoint problem \eqref{Adj-Eq}, and $\alpha = \alpha(t)$ is given by 
\begin{equation}
	\label{Example-Alpha}
	\alpha(t) = 1+ \frac{1}{2} \sin(\omega t), \quad t \in \R, 
\end{equation}
where $\omega \in \R^*$ is suitably chosen.

Similarly as in \cite{ZuaSwitch}, the main difficulty will be to guarantee that for any minimizer $Z_T$ of $J$ (in a suitable class to be defined later), the set $\{ t \in (0,T), \, \| B_1^* Z(t) \|_{U_1}^2 = \alpha (t) \| B_2^* Z(t)\|_{U_2}^2 \}$ is of measure zero, thus guaranteeing the switching structure of the controls provided that way, or corresponds to the straightforward case $Z_T = 0$, see Section \ref{Sec-Proof-Main} for more details. 

As it turns out, this property will strongly use the analyticity of the semigroup of generator $-A^*$. But more than that, we will strongly use the fact that $\alpha$ is analytic and oscillates at infinity and therefore no resonances effect preventing from the switching structure \eqref{Def-Switching-Controls} can arise. 

Before going further, let us remark that the work \cite{ZuaSwitch} proposes a similar strategy, see \cite[p.94 - 95 and Theorem 2.2]{ZuaSwitch}, but did not manage {to conclude that  the set $\{ t \in (0,T), \, \| B_1^* Z(t) \|_{U_1}^2 = \alpha (t) \| B_2^* Z(t)\|_{U_2}^2 \}$ is either of zero measure or corresponds to the trivial case $Z_T = 0$ in the general set up we propose: There,} only the finite dimensional case was considered and it was assumed  that $B_1$ and $B_2$ were scalar (\emph{i.e.} $U_1 = U_2 = \R$) and that $(A, B_1 - \alpha_- B_2)$ and $(A, B_1 + \alpha_+ B_2)$ satisfy Kalman rank conditions for some $\alpha_-$ and $\alpha_+$ in the accumulation sets of $\alpha$ at $-\infty$ and $+\infty$  respectively. Some extensions were given in some particular infinite dimensional settings, but under strong spectral assumptions. 

Here, our arguments avoid these strong spectral requirements by fully using the analytic function $\alpha = \alpha(t)$ in \eqref{Def-J-H} and the fact, that for $\alpha$ of the form \eqref{Example-Alpha},  {the set of}  accumulation points at $-\infty$ is a non-trivial interval.

\begin{remark}
	Let us also point out that this result is easy to obtain in finite dimensional settings, as it was mentioned to us by Marius Tucsnak, who we hereby thank. Indeed, when $H$ is of finite dimension, it is easy to check that for all $T>0$, for any $i \in \{1, 2\}$, considering any non-empty open time interval $I_i$, the set $R_i(I_i)$ defined by 
	$$
		R_i(I_i) = \left\{ \int_0^T e^{ -(T -s) A} B_i {\bf 1}_{I_i}(s) u_i(s) \, ds, \, \text{ with } u_i \in L^2(0,T) \right\}, 
	$$
	\emph{i.e.} the reachable set for \eqref{Main-Eq-u-1-2} at time $T$ starting from $y_0 = 0$ and with controls $u_i$ acting only in the time interval $I_i$ (${\bf 1}_{I_i} $ is the indicator function of the interval $I_i$), the other control being  {null}, equals to the set $\mathcal{R}_i$ defined by
	$$
	 	\mathcal{R}_i = \hbox{Ran\,} (B_i, AB_i, \cdots, A^{d-1} B_i ), 
	$$
	where $d$ is the dimension of the space $H$. In particular, $R_i(I_i)$ is independent of the choice of the time interval $I_i$. 
	
	Recall then that if  $H$ {is a finite dimensional space} of dimension $d$ and system \eqref{Main-Eq-u} is controllable, the Kalman rank condition is satisfied, \emph{i.e.} $ \text{Ran\,} (B, AB, \cdots, A^{n-1} B ) = \R^d$, so that by construction (recall \eqref{Def-B1-B2}) $\mathcal{R}_1 + \mathcal{R}_2 = \R^d$. 
	
	Therefore, using the above comments, given any initial datum $y_0 \in H$ and non-empty open time sub-intervals $I_1$ and $I_2$ of $(0,T)$, there exists controls $u_1 \in L^2(0,T; U_1)$ and $u_2 \in L^2(0,T; U_2)$ such that the solution $y$ of \eqref{Main-Eq-u-1-2} satisfies \eqref{Null-Cont-Req} while $u_1$ is supported in $I_1$ and $u_2$ is supported in $I_2$.
	
	Even if this is a stronger statement than the one of Theorem \ref{Thm-Main} in the case of finite dimension, our approach has the advantage of building a strategy which naturally constructs switching controls, optimizing the choice of the switching times, while the above result would give switching structures through a priori choices of the supports of the controls $u_1$ and $u_2$.
\end{remark}

In fact, our proofs can be adapted to the case of more than $2$ control operators and to unbounded control operators $B \in \mathscr{L}(U, \mathscr{D}(A^*)')$. Assume that $U$ is isomorphic to $U_1 \times \cdots \times U_n$ for some $n \in \N^*$ satisfying $n \geq 2$, \emph{i.e.} 
\begin{equation}
	\label{Link-U1-2-3-U}
	\hbox{there exists a linear isomorphism } \pi : U_1 \times \cdots \times U_n \to U, 
\end{equation}
so that we can associate to $B \in  \mathscr{L}(U, \mathscr{D}(A^*)')$ $n$ operators $B_i \in \mathscr{L}(U_i, \mathscr{D}(A^*)')$, $i \in \{1, \cdots,n\}$, by the formula
\begin{equation}
	\label{Def-B1-B2-B3}
	\forall (u_1, \cdots, u_n) \in U_1 \times \cdots \times U_n,  \quad B \pi (u_1, \cdots, u_n) = \sum_{i = 1}^n B_i u_i.
\end{equation}
When having $n$ controls $u_i \in L^2(0,T; U_i)$, the interesting notion of \emph{switching control} is then the following: 
\begin{equation}
	\label{Def-Switching-3}
	\hbox{a.e. in } t \in (0,T), \,
				\ds 
				\quad
				\prod_{i =1}^n \left( \sum_{j \neq i} \| u_j(t)\|_{U_j} \right) = 0. 
\end{equation}
In other words, we shall say that the controls $(u_1, \cdots, u_n) \in L^2(0,T; U_1 \times \cdots \times U_n)$ are switching if  almost everywhere in $t \in (0,T)$, {at most} one control is active. 

We then claim that Theorem \ref{Thm-Main} can be generalized to this case: 
\begin{theorem}
	\label{Thm-Main-Gal}
	Let us assume one of the two conditions: 
	\begin{itemize}
		\item $A: \mathscr{D}(A) \subset H \to H$ is a self-adjoint positive definite operator with compact resolvent, $H$ being an Hilbert space; 			
		\item $H$ is a finite dimensional vector space.
	\end{itemize}
	Let $B \in \mathscr{L}(U,\mathscr{D}(A^*)')$, where $U$ is an Hilbert space, let $n \in \N$ with $n \geq 2$, and assume that $U$ is isomorphic to $U_1 \times \cdots \times U_n$ for some Hilbert spaces $U_i$, $i \in \{1, \cdots, n\}$, and define $B_i$ for $i \in \{1, \cdots, n\}$ as in \eqref{Def-B1-B2-B3}.

	We assume that system \eqref{Main-Eq-u} is null-controllable in arbitrary small times.

	Then the system 
	\begin{equation}
		\label{Main-Eq-u-1-2-3} 
		y' + A y = \sum_{i = 1}^n B_i u_i, \quad t \in (0,T), \qquad y(0) = y_0,  
	\end{equation}
	is null-controllable in arbitrary small times with switching controls, \emph{i.e.} {controls} satisfying \eqref{Def-Switching-3}. To be more precise, given any $T>0$ and any $y_0 \in H$, there exist $n$ control functions $u_i \in L^2(0,T; U_i)$, $i \in \{1, \cdots, n\}$, such that the solution $y$ of \eqref{Main-Eq-u-1-2-3} satisfies \eqref{Null-Cont-Req} while the control functions satisfy the switching condition \eqref{Def-Switching-3}.
\end{theorem}
The proof of Theorem \ref{Thm-Main-Gal} is given in Section \ref{Sec-Thm-Gal} and follows the same steps as the one of Theorem \ref{Thm-Main}.

We will then give several examples of applications, in particular regarding general parabolic systems and Stokes problem, in Section \ref{Sec-Example}.
We shall also explain under which assumptions Theorems \ref{Thm-Main} and \ref{Thm-Main-Gal} can be extended {to non-self adjoint operators $A$} with compact resolvent which generates an analytic semigroup, see Section \ref{Sec-Extensions} and Theorem \ref{Thm-Ext}. However, it is important to point out immediately that the assumptions required to deal with non-self adjoint operators seem quite delicate to check in practice, as we will explain in two examples, due to the possible complexity of the spectrum in those cases. 

\paragraph{Related results.}

As said above, this work is strongly related to the work \cite{ZuaSwitch}, which triggered our analysis. But more generally, it is related to the common idea that minimizing $\ell^1$ norms enforces sparsity. This idea has been developed thoroughly in the context of optimal control, see e.g. \cite{Alt-Schneider,Ito-Kunisch-2014, Kalise-Kunisch-Rao-2017,Kalise-Kunisch-Rao-2018} and references therein.

As we will se later in the examples described in Section \ref{Sec-Example}, when considering parabolic systems or Stokes problem, Theorem \ref{Thm-Main-Gal} will easily provide controllability results with controls having at each time {at most} one active component. This is in sharp contrast with the questions addressed for parabolic systems or Stokes models when the control can act on only one component, in which the controllability properties can be strongly modified depending on the geometry of the domains or the time of controllability, see e.g. \cite{Ammar-Khodja-et-al-Survey-2011,Ammar-Khodja-et-al-2016,DuprezLissy} and the references therein, while the use of non-linear terms may help to re-establish control properties, see e.g. the works \cite{Chowdhury-Erv-2019, Coron-Guerrero-09,Coron-Lissy}. In other words, the notion that we are analyzing in this context truly lies in between the notions of controllability with controls acting on all components and controllability with controls acting only on one component.

\paragraph{Acknowledgements.} We deeply thank Franck Boyer for his encouragements regarding this work, and we thank him especially for having pointed out to us the example \eqref{Ex-Boyer} provided in Section \ref{Sec-Extensions}.
%
%%%%%%%%%%%%%%%%%%%%%%%%%%%%%%%%%%%%%%%%%%%
%
\section{Proof of Theorem \ref{Thm-Main}}
\label{Sec-Proof-Main}
The structure of the proof of Theorem \ref{Thm-Main} is exactly the same whether $A$ is a self-adjoint operator or $H$ is a finite dimensional space, and strongly follows the one presented in \cite{ZuaSwitch}.

Let $y_0 \in H$ and $T>0$ be fixed, and then introduce the functional $J$ defined in \eqref{Def-J-H} for $z_T \in H$ and $z$ solving \eqref{Adj-Eq}. 

Since $\inf \alpha = 1/2 > 0$ and $\sup \alpha = 3/2 < \infty$, it is clear that the observability property \eqref{Obs-T=0} implies that for all $T>0$, there exists a constant $C_T$ such that for all $z_T \in H$, 
\begin{equation}
	\label{Obs-T=0-B12-alpha}
	\| z(0)\|_H^2 \leq C_T^2 \int_0^T \max\{\| B_1^* z(t) \|_{U_1}^2, \alpha(t) \| B_2^* z(t) \|_{U_2}^2 \} \, dt.
\end{equation}
Although the functional $J$ in \eqref{Def-J-H} is convex, the functional $J$ is in general not coercive with respect to the norm of $H$ (this is for instance the case when considering the heat equation). We thus introduce the space 
\begin{equation}
	\label{Def-X-Space}
	X = \overline{ H }^{\| \cdot \|_{obs}},
\end{equation}
\emph{i.e.} the completion of the space $H$ with respect to the norm $\| \cdot \|_{obs}$ given by 
\begin{equation}
	\label{Norm-B12-alpha}
	\| z_T \|_{obs}^2 =  \int_0^T \max\{\| B_1^* z(t) \|_{U_1}^2, \alpha(t) \| B_2^* z(t) \|_{U_2}^2 \} \, dt.
\end{equation}
One then easily checks that, since this norm is equivalent to 
$$
	\int_0^T \| B^* z(t)\|_U^2 \, dt, 
$$
for $\alpha$ of the form \eqref{Example-Alpha}, the space $X$ does not depend on the choice of the parameter $\omega$ in \eqref{Example-Alpha}.

Using \eqref{Obs-T=0-B12-alpha}, it is clear that the functional $J$ in \eqref{Def-J-H} admits a unique extension (still denoted the same)  as a continuous functional in $X$, and that it is coercive in $X$, and stays convex. 

The functional $J$ has therefore a minimizer $Z_T \in X$. To derive the Euler-Lagrange equation satisfied by $Z_T$, it is convenient to first analyze when the set
\begin{equation}
	\label{Def-I-Switching-Times}
	I = \{ t \in (0,T), \, \| B_1^* Z(t) \|_{U_1}^2 = \alpha(t) \| B_2^* Z(t) \|_{U_2}^2 \} 
\end{equation}
is of non-zero measure. We claim that this can happen only in the straightforward case $Z_T = 0$ in the two cases we are interested in:
\begin{lemma}
	\label{Lemma-A-SelfAdj}
	When $A$ is a self adjoint positive definite operator with compact resolvent and $\alpha$ is as in \eqref{Example-Alpha} with $\omega \in \R \setminus \{0\}$, the set $I$ is necessarily of zero measure, except in the case $ \| B_1^* Z \|_{L^2(0,T;U_1)} = \| B_2^* Z \|_{L^2(0,T; U_2)}  =0$ where $I = (0,T)$.
\end{lemma}
\begin{lemma}
	\label{Lemma-H-Finite-Dim}
	Let $H$ be a finite-dimensional space. Let $(\lambda_k)_{k \in \{1, \cdots, K\}}$ be the eigenvalues of the matrix $A^*$ ordered so that $\Re(\lambda_k) \leq \Re(\lambda_{k+1})$ for all $k$ and define the set $W$ as follows: 
	\begin{equation}
		\label{Def-W}
		W = \{0\} 
%	\\
		\cup \left\{ \Im (\lambda_{k}) - \Im(\lambda_{k_1}), \frac{1}{2}(\Im (\lambda_{k}) - \Im(\lambda_{k_1})) , \, \hbox{ for all } (k,k_1) \hbox{ such that }   \Re (\lambda_{k})=  \Re(\lambda_{k_1}) \right\}.
	\end{equation}
	Then, for $\alpha$ as in \eqref{Example-Alpha} with $\omega \in \R \setminus W$, the set $I$ is necessarily of zero measure, except in the trivial case $ \| B_1^* Z \|_{L^2(0,T;U_1)} = \| B_2^* Z \|_{L^2(0,T; U_2)}  =0$ where $I = (0,T)$.
\end{lemma}
The proofs of Lemma \ref{Lemma-A-SelfAdj}, respectively Lemma \ref{Lemma-H-Finite-Dim}, are postponed to Section \ref{Subsec-Proof-Lem-SelfAdj}, respectively Section \ref{Subsec-Proof-Lem-H-Finite-Dim}. 

\begin{remark}\label{Rem-Lemma-and-UC}
	We point out that Lemma \ref{Lemma-A-SelfAdj} and Lemma \ref{Lemma-H-Finite-Dim} do not use the unique continuation property $ \| B_1^* Z \|_{L^2(0,T;U_1)} = \| B_2^* Z \|_{L^2(0,T; U_2)}  =0$ implies that $Z = 0$ in $(0,T)$, but only the analyticity of the semigroup and the clear structure of the spectrum of the operator $A$ when $A$ is a matrix or a self-adjoint operator.  	This will be of interest in extending Theorem \ref{Thm-Main} to $n$ operators, see Section \ref{Sec-Thm-Gal}.
\end{remark}

Based on the above results, using the observability property \eqref{Obs-T=0}, we deduce that the set $I$ is of zero measure except in the trivial case $Z_T = 0$. Therefore, when $Z_T \neq 0$ setting 
\begin{align*}
	I_1 &= \{ t \in (0,T), \, \| B_1^* Z(t) \|_{U_1}^2 > \alpha(t) \| B_2^* Z(t) \|_{U_2}^2 \}, 
	\\
	\hbox{ and } 
	\\
	\quad
	I_2 &= \{ t \in (0,T), \, \| B_1^* Z(t) \|_{U_1}^2 < \alpha(t) \| B_2^* Z(t) \|_{U_2}^2 \}, 
\end{align*}
the Euler-Lagrange equation satisfied by $Z$ easily yields that for all $z_T \in H$, 
\begin{equation}
	\label{Euler-Lag-Eq}
	0 
	= 
	\int_{I_1} \langle B_1^* Z(t),  B_1^* z(t) \rangle_{U_1} \, dt 
	+ 
	\int_{I_2} \alpha(t) \langle B_2^* Z(t),  B_2^* z(t) \rangle_{U_2} \, dt 
	+ 
	\langle y_0, z(0) \rangle_H, 
\end{equation} 
see \cite[p.91--93]{ZuaSwitch} for the careful justification of this identity.

It is then easy to check that, setting 
\begin{equation}
	\label{Def-Controls}
	u_1(t) 
	= 
		\left\{ 
			\begin{array}{ll}
				 B_1^* Z(t) \, &\hbox{ for }t \in I_1, 
				 \\
				 0\, &\hbox{ for }t \in I_2, 
			\end{array}
		\right.
	\quad
	\hbox{ and }
	\quad
	u_2(t) 
	= 
		\left\{ 
			\begin{array}{ll}
				 0  \, &\hbox{ for }t \in I_1, 
				 \\
				 \alpha(t) B_2^* Z(t)\, &\hbox{ for }t \in I_2, 
			\end{array}
		\right.
\end{equation}
the corresponding solution $y$ of \eqref{Main-Eq-u-1-2} satisfies \eqref{Null-Cont-Req}, while $u_1$ and $u_2$ satisfy the switching condition \eqref{Def-Switching-Controls}.

On the other hand, it is easy to check that, if $Z_T = 0$, then $y_0 = 0$ and the controls $u_1 = 0$ and $u_2 = 0$ are also suitable to control the trajectory \eqref{Main-Eq-u-1-2} to zero at time $T$ (\emph{i.e.} \eqref{Null-Cont-Req}), and obviously satisfy the switching condition \eqref{Def-Switching-Controls}.

It therefore remains to show Lemma \ref{Lemma-A-SelfAdj} and Lemma \ref{Lemma-H-Finite-Dim}, whose proofs are given in the next sections.

\subsection{Proof of Lemma \ref{Lemma-A-SelfAdj}: The case of a self-adjoint positive definite operator $A$ with compact resolvent}
\label{Subsec-Proof-Lem-SelfAdj}
In order to prove that the set $I$ is of zero measure except when $ \| B_1^* Z \|_{L^2(0,T;U_1)} = \| B_2^* Z \|_{L^2(0,T; U_2)}  =0$, we will consider a strictly positive and strictly increasing sequence $T_n$ going to $T$ as $n \to \infty$ and we will show that for all $n \in \N$, the set 
\begin{equation}
	\label{Def-I-n}
	I_n = I \cap (0, T_n)
\end{equation}
is of zero measure except in the trivial case where $B_1^* Z$ and $B_2^* Z$ vanish identically on $(0,T_n)$. This will entail that $I$ is of zero measure as well except in the trivial case where $B_1^* Z$ and $B_2^* Z$ vanish identically.

Let then $n \in \N$ corresponding to $T_n$. From \eqref{Obs-T=0} applied between $T_n$ and $T$, there exists $C_n$ such that for all $z_T \in H$, the solution $z$ of \eqref{Adj-Eq} satisfies
$$
	\sup_{t \in (0,T_n)} \| z(t)\|_{H} \leq C_n \| B^* z\|_{L^2(0, T; U)} \leq {\sqrt{2}}C_n \| z_T\|_{obs}. 
$$
Therefore, the map $z_T \in H \mapsto z(t) \in L^\infty(0,T_n;H)$ extends by continuity to $X$, and in particular, for $Z_T \in X$, 
\begin{align}
	& Z \hbox{ is well-defined on }\, (0,T_n)\, \hbox{ with values in } H, 
	\hbox{ and } 
	\notag
	\\
	& - Z' + A^* Z = 0, \quad t \in (0,T_n), \qquad Z(T_n) = Z_n \in H, 
	\label{Eq-Z-n}
\end{align}
and the set $I_n$ can be equivalently defined as
$$
	I_n = \{ t \in (0,T_n), \, \| B_1^* Z(t) \|_{U_1}^2  = \alpha(t) \| B_2^* Z(t) \|_{U_2}^2 \}. 
$$
Now, since $Z$ satisfies \eqref{Eq-Z-n}, $Z$ is an analytic function on $(0,T_n)$ because $- A^* = -A$ is the generator of an analytic semigroup, and it can thus be extended uniquely as an analytic function on $(- \infty, T_n)$ as the solution of 
\begin{equation}
	\label{Eq-Z-n-infinity}
	- Z' + A^* Z = 0, \quad t \in (-\infty,T_n), \qquad Z(T_n) = Z_n \in H.
\end{equation}
Therefore, since $\alpha$ also is an analytic function, if $I_n$ is of positive measure, then 
\begin{equation}
	\label{Infinite-Time-Identity}
	\forall t \in (-\infty, T_n),\quad \| B_1^* Z(t) \|_{U_1}^2  = \alpha(t) \| B_2^* Z(t) \|_{U_2}^2
\end{equation}
Our next goal is to prove that \eqref{Infinite-Time-Identity} cannot be satisfied except in the trivial case $ \| B_1^* Z \|_{L^2(0,T_n;U_1)} = \| B_2^* Z \|_{L^2(0,T_n; U_2)}  =0$. We thus assume \eqref{Infinite-Time-Identity}. 

Now, since $A$ is a positive definite self-adjoint operator with compact resolvent, its spectrum is given by a positive strictly increasing sequence of eigenvalues $0 < \lambda_1 < \lambda_2 < \cdots < \lambda_k < \lambda_{k+1} \to \infty$ and of corresponding eigenspace $H_k = \hbox{Kernel} ( A - \lambda_k I)$, which are two by two orthogonal.

We expand $Z_n \in H$ using this basis, 
\begin{equation}
	Z_n = \sum_{k \in \N} w_k, \quad \hbox{ with } w_k \in H_k, 
	\quad \hbox{ and } \quad \| Z_n \|_H^2 = \sum_{k} \|w_k\|_H^2,
\end{equation}
so that
\begin{equation}
	\label{Exp-Z(t)}
	\forall t < T_n, \quad Z(t) = \sum_{k \in \N} w_k e^{ \lambda_k (t - T_n)} .
\end{equation} 
Now, let 
\begin{equation}
	\label{Def-k-0}
	k_0 = \inf \{ k \in \N, \, \|B_1^* w_k\|_{U_1} + \| B_2^* w_k \|_{U_2} \neq 0 \}.
\end{equation}
Our goal is thus to check that $k_0$ cannot be finite. If $k_0$ is finite, then we should have
\begin{equation}
	\label{Cond-k-0}
	\|B_1^* w_{k_0}\|_{U_1} + \| B_2^* w_{k_0} \|_{U_2} \neq 0.
\end{equation}
Therefore, setting
$$
	Z_r(t) = \sum_{ k \neq k_0} w_k e^{ \lambda_k (t - T_n)}, \qquad (t < T_n), 
$$
the identity \eqref{Infinite-Time-Identity} implies that for all $t < T_n$, 
\begin{align*}
	\| B_1^* w_{k_0} \|_{U_1} ^2 - \alpha(t) \|B_2^* w_{k_0}\|_{U_2}^2 
	= & 
	- 2 e^{- \lambda_{k_0} (t-T_n)} \Re\left( \langle B_1^* w_{k_0}, B_1^* Z_r(t) \rangle_{U_1} - \alpha(t) \langle B_2^* w_{k_0}, B_2^* Z_r(t) \rangle_{U_2}\right)
	\\
	&
	- e^{-2 \lambda_{k_0} (t-T_n) } \left( \| B_1^* Z_r(t) \|_{U_1} ^2 - \alpha(t) \|B_2^* Z(t)\|_{U_2}^2  \right).
\end{align*}
Since
$$
	\exists C>0, \, \forall t < T_n, \quad \| B_1^* Z_r(t) \|_{U_1} + \| B_2^* Z_r(t)\|_{U_2} \leq C e^{ \lambda_{k_0+1} (t- T_n)}, 
$$
the last identity yields, for all $t < T_n$, 
$$
	\left| \| B_1^* w_{k_0} \|_{U_1} ^2 - \alpha(t) \|B_2^* w_{k_0}\|_{U_2}^2 \right|
	\leq C e^{(\lambda_{k_0+1} - \lambda_{k_0})(t- T_n)} + C e^{2 (\lambda_{k_0+1} - \lambda_{k_0})(t- T_n)}.
$$
Since $\lambda_{k_0+1} > \lambda_{k_0}$, making $t \to -\infty$, we obtain that, 
\begin{equation}
	\forall \alpha_\infty \in [\liminf_{t \to -\infty} \alpha, \limsup_{t \to -\infty} \alpha], 
	\quad
	 \| B_1^* w_{k_0} \|_{U_1} ^2 - \alpha_\infty \|B_2^* w_{k_0}\|_{U_2}^2 = 0. 
\end{equation}
Since $\liminf_{t \to -\infty} \alpha < \limsup_{t \to \infty} \alpha$, we easily get that this implies
$$
	B_1^* w_{k_0} = 0, \quad \hbox{ and } \quad B_2^* w_{k_0} = 0. 
$$
This contradicts \eqref{Cond-k-0}, so that $k_0$ is infinite and thus, $B_1^* Z = 0$ and $B_2^* Z = 0$ on $(-\infty, T_n)$. This shows that, except when $B_1^* Z$ and $B_2^* Z$ vanish identically on $(0,T_n)$, $I_n$ is of zero measure. In particular, passing to the limit $n \to \infty$, we easily get that $I$ is of zero measure except if $B^*_1 Z$ and $B_2^* Z$ vanish identically on $(0,T)$.
\begin{remark}
	In the above proof, we did not really use the specific form of $\alpha$. In fact, as one can check, the proof of Lemma \ref{Lemma-A-SelfAdj} works for any function $\alpha$ satisfying 
	\begin{equation}
		\label{Conditions-on-alpha}
		\left\{
		\begin{array}{l}
			\ds \alpha \hbox{ is an analytic function on } \R, 
			\smallskip\\
			\ds 0 < \inf_\R \alpha  < \sup_\R \alpha < \infty, 
			\\
			\ds \liminf_{t \to - \infty} \alpha < \limsup_{t \to -\infty} \alpha. 
		\end{array}
		\right.
	\end{equation}
\end{remark}
\subsection{Proof of Lemma \ref{Lemma-H-Finite-Dim}: The case of a finite-dimensional space $H$}
\label{Subsec-Proof-Lem-H-Finite-Dim}

In order to prove Lemma \ref{Lemma-H-Finite-Dim}, we will use the following result: 
\begin{lemma}
	\label{Lem-a-j=0}
	Let $J$ be a finite set and $(\mu_j)_{j \in J}$ be a finite sequence of two by two distinct real numbers.
	\\	
	Then, for any finite sequence $(a_j)_{j \in J}$ of elements of $\C$ such that 
	\begin{equation}
		\label{Condition-a-j}
		 \lim_{t \to - \infty} \left( \sum_{j \in J} a_j e^{ i \mu_j t} \right) = 0, 
	\end{equation}
	we have 
	\begin{equation}
		\forall j \in J, \quad a_j = 0.
	\end{equation}
\end{lemma}
\begin{proof}
	To prove Lemma \ref{Lem-a-j=0}, we use that  since there is a finite number of $\mu_j$, 
	$$
		\int_0^1 \left| \sum_{j \in J} b_j e^{i \mu_j t} \right|^2 \, dt 
	$$
	is a norm on $\{ b = (b_j)_{j \in J}, \, b_j \in \C\}$, and is thus equivalent to the quantity 
	$$
		\sum_{j \in J} |b_j|^2. 
	$$
	Now, for $(a_j)_{j \in J}$ as in \eqref{Condition-a-j}, we have, for any $T \in \R$, 
	$$
		\sum_{j \in J} |a_j|^2
		= 
		\sum_{j \in J} | a_j e^{- i \mu_j T} | ^2
		\leq 
		C  \int_0^1 \left| \sum_{j \in J} a_j e^{ i \mu_j (t- T)} \right|^2 \, dt 
		\leq 
		 C\int_{-T}^{-T+1} \left| \sum_{j \in J} a_j e^{ i \mu_j t } \right|^2 \, dt. 
	$$
	Thus, choosing $T$ going to {$+\infty$}, the assumption \eqref{Condition-a-j} and the above estimates give Lemma \ref{Lem-a-j=0}.
\end{proof}
Let us now come back to the proof of Lemma \ref{Lemma-H-Finite-Dim}. To start with, we put the matrix $A^*$ into its Jordan form, and call $(\lambda_k)_{k \in \{1, \cdots, K\}}$ its eigenvalues ordered so that $\Re(\lambda_k) \leq \Re(\lambda_{k+1})$ for all $k$, and we call $H_k$ the corresponding generalized eigenspaces. 

%We then define $W$ as follows: 
%%
%\begin{equation}
%	\label{Def-W}
%	W = \{0\} 
%%	\\
%	\cup \left\{ \Im (\lambda_{k}) - \Im(\lambda_{k_1}), \frac{1}{2}(\Im (\lambda_{k}) - \Im(\lambda_{k_1})) , \, \hbox{ for all } (k,k_1) \hbox{ such that }   \Re (\lambda_{k})=  \Re(\lambda_{k_1}) \right\}.
%\end{equation}
%
We then prove that when $\omega \in \R \setminus W$ (recall the definition of $W$ in \eqref{Def-W}), with the choice of $\alpha$ as in \eqref{Example-Alpha}, $I$ necessarily is of zero measure except in the trivial case $ \| B_1^* Z \|_{L^2(0,T;U_1)} = \| B_2^* Z \|_{L^2(0,T; U_2)}  =0$. 

We thus assume that $I$ is of non-zero measure and we let $Z$ be the solution of \eqref{Adj-Eq} with initial datum $Z_T \in X$. Here, since $H$ is finite dimensional, $X = H$ and $Z_T \in H$. Then the solution $Z$ of \eqref{Adj-Eq}  can be defined on $\R$, is an analytic function of time, and we write it under the form 
\begin{equation}
	\label{Expansion-Z-t}
	Z(t) = \sum_{ k} e^{\lambda_k (t - T)} \left( \sum_{\ell = 0}^{m_k} (T -t)^\ell w_{k,\ell} \right), \qquad (t \in \R),
\end{equation}
where $m_k$ is the size of the maximal Jordan block corresponding to $\lambda_k$ (or equivalently, its algebraic multiplicity), and each $w_{k, \ell}$ belongs to $H_k$. Besides, since we assume that $I$ is of non zero measure and since $Z$ in \eqref{Expansion-Z-t} is analytic with respect to time, we should have {$I = (0,T)$, and it follows that}
\begin{equation}
	\label{I=R-Identity}
	\forall t \in \R, \quad \| B_1^* Z(t) \|_{U_1}^2 = \alpha(t) \| B_2^* Z(t)\|_{U_2}^2.
\end{equation}

Now, let 
$$
	k_0 = \inf \left\{ k \in \{1, \cdots, K\}: \  \exists \ell \in \{0, \cdots, m_k\} \ {\text{such that}}  \   \| B_1^* w_{k, \ell} \|_{U_1}+ \| B_2^* w_{k, \ell} \|_{U_2} \neq 0 \right\}.
$$
If {$k_0 <\infty $}, we define $\ell_1$ by
$$
	\ell_1 = \sup \left\{\ell: \, {\exists} k \hbox{ with } \Re(\lambda_{k}) = \Re(\lambda_{k_0})  \hbox{ and }  \| B_1^* w_{k, {\ell}} \|_{U_1}+ \| B_2^* w_{k, {\ell}} \|_{U_2} \neq 0 \right\}, 
$$
and the set 
$$
	D = \left \{ k:  \ \Re(\lambda_{k}) = \Re(\lambda_{k_0}) \hbox{ and } \| B_1^* w_{k, \ell_1} \|_{U_1}+ \| B_2^* w_{k, \ell_1} \|_{U_2} \neq 0 \right \}, 
$$
which describes the indices giving the dominant terms in $\| B_1^* Z(t) \|_{U_1}^2 - \alpha(t) \| B_2^* Z(t)\|_{U_2}^2$ as $t \to -\infty$. Indeed, setting 
\begin{equation}
	\label{Def-Z-d-Z-r}
	Z_d(t) =  \sum_{k \in D} w_{k,\ell_1} e^{i \Im(\lambda_k) (t-T)}, 
	\quad
	\hbox{ and }
	\quad
	Z_r(t) = Z(t) - e^{{\Re (\lambda_{k_0})} (t-T)} (T-t)^{\ell_1} Z_d(t), 
\end{equation}
we have, for some $C$ independent of time
\begin{equation}
	\label{Estimate-B-Z-r-0}
	\forall t \in (- \infty, T), \quad  \| B_1^* Z_r(t) \|_{U_1}+ \| B_2^* Z_r(t) \|_{U_2} \leq 
	{ 
	\left\{
		\begin{array}{ll}
		C e^{\Re(\lambda_{k_0}) t}( 1+ (T - t)^{\ell_1 - 1}) & \text{ if }\ell_1 \geq 1, 
		\\
		C e^{(\Re(\lambda_{k_0+1})+ \Re(\lambda_{k_0})) t/2} & \text{ if } \ell_1 = 0, 
	\end{array}
	\right.
	}
	%C e^{\textcolor{red}{ \Re(\lambda_{k_0})} t} (T - t)^{\ell_1 - 1}
	.  
\end{equation}
{and thus, possibly changing the constant,}
\begin{equation}
	\label{Estimate-B-Z-r}
	{ 
	\forall t \in (- \infty, T-1), \quad  \| B_1^* Z_r(t) \|_{U_1}+ \| B_2^* Z_r(t) \|_{U_2} \leq 
		C e^{\Re(\lambda_{k_0}) t} (T - t)^{\ell_1 - 1} 
	}
	%C e^{\textcolor{red}{ \Re(\lambda_{k_0})} t} (T - t)^{\ell_1 - 1}
	.  
\end{equation}
Therefore, using \eqref{I=R-Identity}, we easily get that 
\begin{equation}
	\label{Est-Z-d}
	\forall t \in (-\infty, T-1), 
	\quad 
	\left| 
		 \| B_1^* Z_d(t) \|_{U_1}^2 - \alpha(t) \| B_2^* Z_d(t)\|_{U_2}^2
	\right| 
	\leq 
	{\frac{C}{T - t} }. 
\end{equation}
Now, we expand $ \| B_1^* Z_d(t) \|_{U_1}^2 - \alpha(t) \| B_2^* Z_d(t)\|_{U_2}^2$: 
\begin{align}
	\notag
	%\lefteqn{
	 \| B_1^* Z_d(t) \|_{U_1}^2
	 & - \alpha(t) \| B_2^* Z_d(t)\|_{U_2}^2
	  =
	 \sum_{k\in D}  \| B_1^* w_{k,\ell_1} \|_{U_1}^2 
	 - 
	\left(1 + \frac{ \sin(\omega t)}{2} \right) \sum_{k\in D} \| B_2^* w_{k,\ell_1}\|_{U_2}^2
	%}
	\\
	\label{Expansion-B-1-B-2} 
	& + 
	2  \sum_{k\in D} \sum_{k_1 \in D, \, k_1 > k} \Re\left( e^{  i( \Im(\lambda_k) - \Im(\lambda_{k_1})) t} \langle B_1^* w_{k,\ell_1}, B_1^* w_{k_1,\ell_1} \rangle_{U_1}\right)
	\\
	\notag
	&- 	
	2 \left(1 + \frac{ \sin(\omega t)}{2} \right)  \sum_{k \in D} \sum_{k_1 \in D, \, k_1 > k} \Re\left( e^{  i( \Im(\lambda_k) - \Im(\lambda_{{k_1}})) t} \langle B_2^* w_{k,\ell_1}, B_2^* w_{k_1,\ell_1} \rangle_{U_2}\right).
\end{align}
From this, we deduce that the function $ \| B_1^* Z_d(t) \|_{U_1}^2 - \alpha(t) \| B_2^* Z_d(t)\|_{U_2}^2$ is of the form $\sum_j a_j e^{ i \mu_j t}$, where 
$$
	\{ \mu_j \} = 
	\{
		0, 
		\,
		\pm \omega, 
		\, 
		 (\Im(\lambda_k) - \Im(\lambda_{k_1})),
		\, 
		\pm \omega + (\Im(\lambda_k) - \Im(\lambda_{k_1})) \hbox{ for } k, k_1 \in D
	  \}.
$$
This set is finite, but there might be some non-distinct values in the set given on the right hand-side. We shall thus rely on the choice $\omega \notin W$ (recall that $W$ is defined in \eqref{Def-W}), which guarantees that $0$ and $\omega$ appears only once in the above list. Therefore, using \eqref{Est-Z-d}, Lemma \ref{Lem-a-j=0} guarantees at least that the numbers in front of the constant term (corresponding to $\mu = 0$) and of $e^{ i \omega t}$ in \eqref{Expansion-B-1-B-2} vanish, \emph{i.e.}:
\begin{align*}
	0 & =   \sum_{k\in D}  \| B_1^* w_{k,\ell_1} \|_{U_1}^2 
	 - 
	 \sum_{k\in D} \| B_2^* w_{k,\ell_1}\|_{U_2}^2, 
	 \\
	 0 & =  \sum_{k\in D} \| B_2^* w_{k,\ell_1}\|_{U_2}^2. 
\end{align*}
Combining the two above identities, we easily deduce that 
$$
	\forall k \in D, \quad  \| B_1^* w_{k, \ell_1} \|_{U_1}+ \| B_2^* w_{k, \ell_1} \|_{U_2}= 0.
$$
In view of the definition of the set $D$, the set $D$ is necessarily empty. This contradicts the definition of $k_0$ and $\ell_1$. Hence, {$k_0 = \infty$} and $B_1^* Z(t) = 0$ and $B_2^*Z(t) = 0$ for all $t \in \R$.
%
%%%%%%%%%%%%%%%%%%%%%%%%%%%%%%%%%%%%%%%%%%%
%
\section{Proof of Theorem \ref{Thm-Main-Gal}}
\label{Sec-Thm-Gal}

Of course, the proof of Theorem \ref{Thm-Main-Gal} follows the one of Theorem \ref{Thm-Main}. We only point out the main differences that are needed in the proof of Theorem \ref{Thm-Main} to conclude Theorem \ref{Thm-Main-Gal}. 

To fix the ideas, we consider only the case $ n = 3$, as the case of $ n \geq 4$ control operators can be done exactly similarly to the price of adding some notations.

	Given $y_0 \in H$, we consider the functional 
	\begin{equation}
		\label{Def-J-H-3-Controls}
		J(z_T) = \frac{1}{2} \int_0^T \max \{ \alpha_1(t) \| B_1^* z(t)\|_{U_1}^2, \alpha_2(t) \| B_2^* z(t)\|_{U_2}^2, \alpha_3 (t) \| B_3^* z(t) \|_{U_3}^2 \} \, dt + \langle y_0, z(0) \rangle_H, 
	\end{equation}
	defined for $z_T \in \mathscr{D}(A^*)$, where $z$ is the solution of the adjoint problem \eqref{Adj-Eq}, and $\alpha_i = \alpha_i(t)$ is given by 
	\begin{equation}
		\label{Example-Alpha-3-controls}
		\alpha_i(t) = 1+ \frac{1}{2} \sin(\omega_i t), \quad t \in \R, \quad i \in \{1, 2, 3\}.
	\end{equation}
	where the frequencies $\omega_i $ are suitably chosen. 

	Similarly as in the proof of Theorem \ref{Thm-Main}, the functional $J$ can be extended by continuity on the space 
$$
		X = \overline{\mathscr{D}(A^*)}^{\|\cdot \|_{obs} },
$$
where the norm $\|\cdot \|_{obs}$ is the one defined by 
	$$
		\| z_T\|_{obs}^2 = \int_0^T \max \{ \alpha_1(t) \| B_1^* z(t)\|_{U_1}^2, \alpha_2(t) \| B_2^* z(t)\|_{U_2}^2, \alpha_3 (t) \| B_3^* z(t) \|_{U_3}^2 \} \, dt, 
	$$
and is coercive on that space $X$. Therefore, $J$ has a  minimizer $Z_T \in X$. Next, to properly derive the Euler-Lagrange equation satisfied by $Z_T$, we study the sets 
	\begin{equation}
		\forall (i,j) \in \{1, 2, 3\}^2 \hbox{ with } i < j, \quad
		I_{i,j} = \left\{ t \in (0,T), \, \alpha_i(t) \|  B_i^* Z(t) \|_{U_i}^2 = \alpha_j(t) \|  B_j^* Z(t) \|_{U_j}^2 \right\}.
	\end{equation}
	\medskip

	{\bf The case $A$ self-adjoint positive definite operator with compact resolvent.} When $A$ is a self-adjoint positive definite operator with compact resolvent, Lemma \ref{Lemma-A-SelfAdj} can be easily adapted to show the following result: 
	
	\begin{lemma}
	\label{Lemma-A-SelfAdj-Gal}
	When $A$ is a self adjoint positive definite operator with compact resolvent and $(\alpha_i)_{i \in \{1, 2, 3\} }$ are as in \eqref{Example-Alpha-3-controls} with $(\omega_1, \omega_2, \omega_3) \in \R_+^3$ two by two distincts, for all $i,j\in \{1, \cdots, 3\}$ with $i \neq j$, the set $I_{i,j}$ is necessarily of zero measure, except in the trivial case $ \| B_i^* Z \|_{L^2(0,T;U_i)} = \| B_j^* Z \|_{L^2(0,T; U_j)}  =0$.
	\end{lemma}
	Since the proof of Lemma \ref{Lemma-A-SelfAdj-Gal} is completely similar to the {proof of Lemma \ref{Lemma-A-SelfAdj}}, relying on the fact that $\alpha_i/\alpha_j$ admits a set of accumulation points at $-\infty$ which contains a non-trivial interval, we skip it and leave it to the reader.	
	\medskip
	
	{\bf The case $H$ of finite dimension.} When $H$ is a finite dimensional vector space, we choose the parameters $\omega_i$ successively, for instance we can take 
	\begin{equation}
		\label{Def-om-1-2}
		\omega_1 = 0, \quad \omega_2 \in \R \setminus W, 
	\end{equation}
	where $W$ is defined in \eqref{Def-W}, and 
	\begin{equation}
		\label{Def-om-3}
		\omega_3 \in \R \setminus W_3, 
	\end{equation}
	where $W_3$ is defined by
	\begin{equation}
		 \label{Def-W-3}
		W_3 = W \cup \{\pm \omega_2, \pm \omega_2 +  \Im(\lambda_k) - \Im(\lambda_{k_1} ) 
		, \hbox{ for all $(k,k_1)$ such that $\Re(\lambda_k) = \Re(\lambda_{k_1} )$} \}. 
	\end{equation}
	We then prove the following result:
	\begin{lemma}
		\label{Lemma-H-Finite-Dim-3}
	When $H$ is a finite-dimensional space, setting $\omega_1 = 0$, and choosing $\omega_2 \in \R \setminus W$ (defined in \eqref{Def-W}) and $\omega_3 \in \R \setminus W_3$ (defined in \eqref{Def-W-3}) and taking $\alpha_i$ as in \eqref{Example-Alpha} corresponding to $\omega_i$, for all $(i,j) \in \{1, 2, 3\}$ with $i < j$, the set $I_{i,j}$ is necessarily of zero measure, except in the trivial case $ \| B_i^* Z \|_{L^2(0,T;U_i)} = \| B_j^* Z \|_{L^2(0,T; U_j)}  =0$.
	\end{lemma}

	We briefly sketch the proof of Lemma \ref{Lemma-H-Finite-Dim-3} below. 
	
	\begin{proof}[Sketch of the proof of Lemma \ref{Lemma-H-Finite-Dim-3}]
		Clearly, when $i = 1$, the proof of Lemma \ref{Lemma-H-Finite-Dim-3} reduces to the proof of Lemma \ref{Lemma-H-Finite-Dim}.
		
		We thus focus on the case $i = 2$ and $j = 3$. Similarly as in the proof of Lemma \ref{Lemma-H-Finite-Dim}, we assume that $I_{2,3}$ is of positive measure. By analyticity, this implies that $I_{2,3} = (0,T)$, and by extending $Z$ on $\R$ by analyticity, that for all $t \in \R$, $\alpha_2(t) \| B_2^* Z(t) \|_{U_2}^2 = \alpha_3(t) \| B_3^* Z(t) \|_{U_2}^2$. We then expand $Z$ as in \eqref{Expansion-Z-t} and define, as in the proof of Lemma \ref{Lemma-H-Finite-Dim}, 
	\begin{align*}
			k_{0} & = \inf \left\{ k \in \{1, \cdots, K\}:  \, \exists \ell \in \{0, \cdots, m_k\} \  {\text{such that}} \  \| B_2^* w_{k, \ell} \|_{U_2}+ \| B_3^* w_{k, \ell} \|_{U_3} \neq 0 \right\}, 
	\end{align*}
	{and, if $k_0 < \infty$,}
	\begin{align*}
		\\
			\ell_{1} & = \sup \left\{\ell: \, {\exists}  k \hbox{ with } \Re(\lambda_{k}) = \Re(\lambda_{k_{0}})  \hbox{ and }  \| B_2^* w_{k, {\ell}} \|_{U_2}+ \| B_3^* w_{k, {\ell}} \|_{U_3} \neq 0 \right\}, 
		\\
			D & =  \left \{ k:   \Re(\lambda_{k}) = \Re(\lambda_{k_{0}}) \hbox{ and } \| B_2^* w_{k, \ell_1} \|_{U_2}+ \| {B_3^*} w_{k, \ell_1} \|_{U_3} \neq 0 \right \}, 
		\\
			Z_{d}(t) & =  \sum_{k \in D} w_{k,\ell_1} e^{i \Im(\lambda_k) (t-T)}, \qquad (t \in \R).
	\end{align*}
	With the above choices, similarly as in \eqref{Expansion-B-1-B-2}, we have the formula, for all $t \in \R$,  
	\begin{align}
			\notag
		%\lefteqn{
			\alpha_2(t) \| B_2^* Z_d(t) \|_{U_2}^2
			 & - \alpha_3(t) \| B_3^* Z_d(t)\|_{U_3}^2
			 \\
			\notag
			  = &
			\left(1 + \frac{ \sin(\omega_2 t)}{2} \right)  \sum_{k\in D}  \| B_2^* w_{k,\ell_1} \|_{U_2}^2 
			 - 
			\left(1 + \frac{ \sin(\omega_3 t)}{2} \right) \sum_{k\in D} \| B_3^* w_{k,\ell_1}\|_{U_3}^2
			%}
			\\
			\label{Expansion-B-i-B-j} 
			& + 
			2  \left(1 + \frac{ \sin(\omega_2 t)}{2} \right)  \sum_{k\in D} \sum_{k_1 \in D, \, k_1 > k} \Re\left( e^{  i( \Im(\lambda_k) - \Im(\lambda_{k_1})) t} \langle B_2^* w_{k,\ell_1}, B_2^* w_{k_1,\ell_1} \rangle_{U_2}\right)
			\\
			\notag
			&- 	
			2 \left(1 + \frac{ \sin(\omega_3 t)}{2} \right)  \sum_{k \in D} \sum_{k_1 \in D, \, k_1 > k} \Re\left( e^{  i( \Im(\lambda_k) - \Im(\lambda_{k+1})) t} \langle B_3^* w_{k,\ell_1}, B_3^* w_{k_1,\ell_1} \rangle_{U_3}\right). 
	\end{align}
	which stands instead of \eqref{Expansion-B-1-B-2}. Besides, since for all $t \in \R$, we have $\alpha_2(t) \| B_2^* Z(t) \|_{U_2}^2 - \alpha_3(t) \| B_3^* Z(t)\|_{U_3}^2 = 0$, we can also deduce, as in \eqref{Estimate-B-Z-r}, that 
	\begin{equation}
		\label{Est-Z-d-23}
		\forall t \in (-\infty, {T-1}), 
		\quad 
			\left| 
				\alpha_2(t) \| B_2^* Z_d(t) \|_{U_2}^2	- \alpha_3(t) \| B_3^* Z_d(t)\|_{U_3}^2
			\right| 
			\leq 
		\frac{C}{T - t} . 
	\end{equation}	
	Accordingly, using Lemma \ref{Lem-a-j=0} on the function $t \mapsto \alpha_2(t) \| B_2^* Z_d(t) \|_{U_2}^2	- \alpha_3(t) \| B_3^* Z_d(t)\|_{U_3}^2$, which goes to $0$ as $t \to -\infty$, and considering the coefficients in front of the constant term and in front of $e^{i \omega_3t }$ in \eqref{Expansion-B-i-B-j}, which appear only once in the expansion \eqref{Expansion-B-i-B-j} since $\omega_3 \notin W_3$,  we deduce
	\begin{align*}
			0 & =   \sum_{k\in D}  \| B_2^* w_{k,\ell_1} \|_{U_2}^2 
				 - 
				 \sum_{k\in D} \| B_3^* w_{k,\ell_1}\|_{U_3}^2, 
			 \\
			 0 & =  \sum_{k\in D} \| B_3^* w_{k,\ell_1}\|_{U_3}^2. 
	\end{align*}
	{This easily yields that $k_0 = \infty$}, and consequently that $\norm{B_2^* Z(t)}_{U_2} + \norm{B_3^* Z(t)}_{U_3}  = 0$ for all $t \in \R$, and concludes the proof of Lemma \ref{Lemma-H-Finite-Dim-3}.
	\end{proof}
	%
%	%
%	\begin{lemma}
%		\label{Lem-H-FD-3}
%		%
%		When $H$ is a finite-dimensional space and assuming that $A$ satisfies \eqref{Condition-B-i*}, setting $W$ as in \eqref{Def-W} and choosing $\omega_1$, $\omega_2$ and $\omega_3$ as in \eqref{Def-om-1-2}--\eqref{Def-om-3}, taking $\alpha_i$ as in \eqref{Example-Alpha-3-controls}, each  set $I_{i,j}$ for $i \neq j$ is necessarily of zero measure, except in the trivial case $Z_T = 0$
%		%
%	\end{lemma}
%	%
%	We postpone the proof of Lemma \ref{Lem-H-FD-3} afterwards, since it is slightly more involved than in the case of a self-adjoint operator with compact resolvent. 
	%
	\medskip
	
	{\bf End of the proof of Theorem \ref{Thm-Main-Gal}.}
	We choose the coefficients $(\omega_1, \omega_2, \omega_3) \in \R^3$ such that the assumptions of Lemma \ref{Lemma-A-SelfAdj-Gal} are satisfied in the case of a self-adjoint operator, or such that the assumptions of Lemma \ref{Lemma-H-Finite-Dim-3} are satisfied when considering the case of $H$ of finite dimension. According to Lemma \ref{Lemma-A-SelfAdj-Gal} and \ref{Lemma-H-Finite-Dim-3}, if $I_{i,j}$ is of positive measure for some $i,j \in \{1, 2,3\}$ with $i \neq j$, taking $\ell \in \{1, 2, 3\} \setminus \{i,j\}$, only two cases arise: 
	\begin{itemize}
		\item If $t \mapsto \| B_\ell^* Z(t)\|_{U_{\ell}}^2$ is identically zero, the observability property \eqref{Obs-T=0} implies that $Z = 0$ identically, which corresponds to a minimizer for $J$ only in the case $y_0 = 0$, which can be steered to $0$ by keeping all the controls equal to $0$ at all times.
		\item If $t \mapsto \| B_\ell^* Z(t)\|_{U_{\ell}}^2$ is not identically zero, since it is an analytic function, its zero set has no accumulation point and thus 
		$$
			\hbox{a.e. } t \in (0,T), \quad 
			\alpha_{\ell}(t) \| B^*_\ell Z(t) \|_{U_{\ell}}^2
			> 
			\max \{\alpha_i(t) \| B_i^* {Z}(t)\|_{U_i}^2, \alpha_j(t) \| B_j^* {Z}(t)\|_{U_j}^2\}.
		$$
	\end{itemize}
	Accordingly, except in the trivial case $Z_T =0$, we have the following:
	\begin{equation}
		\label{Cond-For-Eul-Lag}
		\hbox{a.e. } t \in (0,T),\, \exists ! \ell \in \{1, 2, 3\}, \hbox{ sucht that }
		\alpha_{\ell}(t) \| B^*_\ell Z(t) \|_{U_{\ell}}^2
			> 
		\max_{i \neq \ell} \{\alpha_i(t) \| B_i^* {Z}(t)\|_{U_i}^2\}.
	\end{equation}
	 We can then write the Euler-Lagrange equation satisfied by a minimizer $Z_T$ of $J$, and obtain that, setting for each $i \in \{1, 2, 3\}$,
	$$
		u_i(t) 
		= 
		\left\{ 
			\begin{array}{ll} 
				\ds\alpha_i(t) B_i^* Z(t) & \hbox{ when } \alpha_i(t) \| B_i^* Z(t) \|_{U_i}^2 > \displaystyle \max_{j \neq i} \{\alpha_j(t) \| B_j^* Z(t) \|_{U_j}^2\}, 
				\\
				\ds \,0  &\hbox{ else},
			\end{array}
		\right.
	$$
	the corresponding solution $y$ of \eqref{Main-Eq-u-1-2-3} satisfies $y(T) = 0$ while the controls $u_1$, $u_2$, $u_3$ satisfy the switching condition \eqref{Def-Switching-3}.

\section{Examples}
\label{Sec-Example}

\subsection{Examples in finite dimension}
\label{Subsec-Examples-Finite-Dim}

Theorems \ref{Thm-Main} and \ref{Thm-Main-Gal} have many interesting consequences even for finite dimensional systems. Let us give below some examples. 

\paragraph{Example 1: General matrix $A$.} Let us fixed $H = \R^d$ for $d \in \N^*$ and $A$ a $d\times d$ matrix. Then it is clear that the control system
\begin{equation}
	\label{Ex-FD-Gal-A}
	y' +A y  = \left( \begin{array}{c} u_1 \\ u_2 \\ \vdots \\ u_d \end{array} \right), {\quad t \in (0,T)}, \qquad y(0) = y_0 \in {\R^d}, 
\end{equation}
is exactly controllable at any time $T$. Indeed, controllability can be achieved as follows: given $y_0$ and $y_1$ in $\R^d$, we take $y$ a smooth function of time with values in $\R^d$ such that $y(0) = y_0$ and $y(T) = y_1$, and simply set $u = y' + A y$. 

Therefore, it is clear that Theorem \ref{Thm-Main-Gal} applies when considering the operators $B_i u_i = u_i e_i$ for $i \in \{1, \cdots, d\}$, where $e_i$ is the vector of $\R^d$ whose $i$-th component equals $1$ and all the others vanish. We thus get the following result: 
\begin{theorem}
	\label{Thm-FD-Gal-A}
	Let $d \in \N^*$, $H = \R^d$ and $A$ a $d\times d$ matrix. Then for any $y_0 \in \R^d$, there exist $d$ control functions $u_i \in L^2(0,T; \R)$ such that the controlled trajectory of \eqref{Ex-FD-Gal-A} satisfies $y(T) = 0$ and with control functions satisfying condition \eqref{Def-Switching-3}, \emph{i.e.} such that almost everywhere in $(0,T)$, at most one of the controls $u_i(t)$ for $i \in \{1, \cdots, d\}$ is non-zero.
\end{theorem}

This result can be applied for instance to the following case, which corresponds to the space semidiscretization of the $1$-d heat equation on $(0,L)$ with homogeneous Dirichlet boundary conditions at $x = 0$ and $x = L$:
\begin{equation}
	\label{Disc-Heat-Equation}
	\left\{ \begin{array}{ll}
	\ds y_j'  - \frac{1}{h^2}( y_{j+1} - 2 y_j + y_{j-1} ) = u_j, \quad &  t \in (0,T),\, j \in \{1, \cdots, d\}, 
	\\
	y_0(t) = y_{d+1}(t) = 0, & t \in (0,T),
	\\
	\ds y_j(0) = y^0_j & j \in \{1, \cdots, d\}, 
	\end{array}
	\right.
\end{equation}
where $h>0$ is a (small) parameter.
Indeed, equation \eqref{Disc-Heat-Equation} can be seen as the finite difference approximation of the heat equation 
\begin{equation}
	\label{1-d-Heat-Equation}
	\left\{ \begin{array}{ll}
	\ds \partial_t y - \partial_{xx} y = u, \quad & t \in (0,T), \, x \in (0,L),
	\\
	y(t,0) = y(t,L) = 0, & t \in (0,T),
	\\
	\ds y(0, x) = y_0(x) & x \in (0,L),
	\end{array}
	\right.
\end{equation}
choosing the parameter $h$ in \eqref{Disc-Heat-Equation} of the form $h = L/(d+1)$. Theorem \ref{Thm-FD-Gal-A} then yields that \eqref{Disc-Heat-Equation} can be controlled to zero with controls $u_i \in L^2(0,T; \R)$ for each $i \in \{1, \cdots, d\}$ such that at any time, only one of the controls $u_i$ is active. 

It is not clear how that process can pass to the limit as $d \to \infty$, and this is an interesting open question. 

%So far, what is clear is that if we fix $n$ points $a_1 = 0 < a_2 < \cdots < a_{n+1} = 1$ corresponding to $n$ intervals $I_i = [a_{i}, a_{i+1})$ for $i \in \{1, \cdots, n\}$, and introducing $B_i$ the operator defined on $L^2(I_i)$ by $B_i u_i = u_i(x) 1_{I_i}(x)$ in $(0,1)$, Theorem \ref{Thm-Main-3-controls} immediately apply for \eqref{Disc-Heat-Equation} for any $d \in \N$ and for \eqref{1-d-Heat-Equation}, and one may hope for convergence of the discrete controls obtained that way.  

\paragraph{Example 2: General matrices $(A,B)$ satisfying Kalman condition.} 

If $A$ is a $d \times d$ matrix and $B$ is a $d \times n$ matrix, it is well known (see \emph{e.g.} \cite{TWBook}) that the system \eqref{Main-Eq-u} is controllable if and only if the following Kalman condition is satisfied: 
\begin{equation}
	\label{Kalman-Cond}
	{\hbox{Rank}} (B \,,  A B\, , \, A^2 B  \, \cdots , \, A^{d-1} B ) = d. 
\end{equation}
Now, we have chosen $B$ under the form of a $d \times n$ matrix, meaning that the control function $u$ belongs to $u \in L^2(0,T; \R^n)$. As before, when $n \geq 2$, it is interesting to write
\begin{equation}
	B u = \sum_{i = 1}^n B_i u_i, \quad \hbox{ where } B_i \hbox{ is the $i$-th column of $B$}.
\end{equation}
Applying then Theorem \ref{Thm-Main-Gal}, we get the following result: 
\begin{theorem}
	\label{Thm-FD-Kalman}
	Let $A$ be a $d \times d$ matrix and $B$ be a $d \times n$ matrix such that the Kalman rank condition \eqref{Kalman-Cond} holds, and let $B_i$ denote the $i$-th column of the matrix $B$. Then for any $y_0 \in \R^d$, there exist $n$ control functions $u_i \in L^2(0,T; \R)$ such that the controlled trajectory of \eqref{Main-Eq-u-1-2-3} satisfies $y(T) = 0$ and with control functions satisfying condition \eqref{Def-Switching-3}, \emph{i.e.} such that almost everywhere in $(0,T)$, at most one of the controls $u_i(t)$ for $i \in \{1, \cdots, d\}$ is non-zero.
\end{theorem}
Again, a nice application is given by the space semi-discretization of some PDE, for instance of the wave equation. Indeed, if we consider the wave equation 
\begin{equation}
	\label{1-d-Wave-Equation}
	\left\{ \begin{array}{ll}
	\ds \partial_{tt} y - \partial_{xx} y = u, \quad & t \in (0,T), \, x \in (0,L),
	\\
	y(t,0) = y(t,L) = 0, & t \in (0,T),
	\\
	\ds (y(0, x), \partial_t y(0,x)) = (y_0(x),y_1(x)) & x \in (0,L), 
	\end{array}
	\right.
\end{equation}
its finite difference semi-discretization is given by 
\begin{equation}
	\label{Disc-Wave-Equation}
	\left\{ \begin{array}{ll}
	\ds y_j''  - \frac{1}{h^2}( y_{j+1} - 2 y_j + y_{j-1} ) = u_j, \quad &  t \in (0,T),\, j \in \{1, \cdots, d\}, 
	\\
	y_0(t) = y_{d+1}(t) = 0, & t \in (0,T),
	\\
	\ds (y_j(0), y_j'(0)) = (y_{j}^0, y_{j}^1) & j \in \{1, \cdots, d\}, 
	\end{array}
	\right.
\end{equation}
where $h = L/(d+1)$. It is clear that the system \eqref{Disc-Wave-Equation} is controllable in any arbitrary time, so that Theorem \ref{Thm-FD-Kalman} applies immediately and provides controls $u_i \in L^2(0,T)$ for all $i \in \{1, \cdots, d\}$ such that at all times only one of the control is active. 

Here again, it is completely unclear how this process can pass to the limit as $d \to \infty$. It is even probably more difficult {to analyze} than in the previous example since the limit equation \eqref{1-d-Wave-Equation} does not correspond to an analytic semigroup. Still, the recent works on sparse optimal controls for the wave equation, see in particular \cite{Kunisch-Trautman-Vexler}, may yield some insights on this problem. 

\subsection{Distributed control of parabolic systems}
\label{Subsec-Ex-1}
To give a non-trivial PDE example, {let us consider $\Omega$ a smooth bounded domain of  $\mathbb{R}^N$ ($N \geq 1$), an open subset $\mathcal{O} \subset \Omega$} and  the following parabolic system: 
\begin{equation}
	\label{ParaSys-Ex1}
	\left\{
		\begin{array}{ll}
			\ds \partial_t y - D \Delta y + P y = 1_\mathcal{O} \left( \begin{array}{c} u_1 \\ u_2 \end{array} \right), & \hbox{ in } (0,T) \times \Omega, 
			\\
			y = 0 & \hbox{ on } (0,T) \times \partial \Omega, 
			\\
			y(0, \cdot) = y_0 & \hbox{ in } \Omega, 
		\end{array}
	\right.
\end{equation}
where 
$$
	y =  \left( \begin{array}{c} y_1 \\ y_2 \end{array} \right), 
	\quad
	D = \left( \begin{array}{cc} d_1 & 0 \\ 0 & d_2 \end{array} \right), \hbox{ with } d_1, d_2 >0, 
$$
and $P = P(x) \in L^\infty(\Omega; S_{2}^+(\R))$, where $ S_{2}^+(\R)$ denotes the set of symmetric positive definite $2\times 2$ matrices with real coefficients. Here, the control 
$$
	u =  \left( \begin{array}{c}u_1 \\ u_2 \end{array} \right), 
$$
acts on the system \eqref{ParaSys-Ex1} on $\mathcal{O}$ through the multiplication by the indicator function $1_\mathcal{O}$ of the subset $\mathcal{O}$.

System \eqref{ParaSys-Ex1} fits into the framework of Theorem \ref{Thm-Main}, by setting 
\begin{equation}
	\label{OpA-Ex1}
	A =  - D \Delta_x + P, \quad \hbox{ in } H = (L^2(\Omega))^2 \hbox{ with domain } \mathscr{D}(A) = (H^2 \cap H^1_0(\Omega))^2,
\end{equation}
and
$$
	B u =  1_{\mathcal{O}} \left( \begin{array}{c} u_1 \\ u_2 \end{array} \right) \hbox{ for } u =  \left( \begin{array}{c} u_1 \\ u_2 \end{array} \right), \quad U = (L^2(\mathcal{O}))^2.
$$
Indeed, the operator $A$ in \eqref{OpA-Ex1} is obviously self-adjoint with compact resolvent. Besides, the following result is a straightforward consequence of the Carleman estimates in \cite{FursikovImanuvilov}:
\begin{proposition}
	\label{Prop-Ex1-Cont}
	System \eqref{ParaSys-Ex1} is null-controllable in arbitrarily small times with control functions $u$ in  $ L^2(0,T; (L^2(\mathcal{O}))^2)$.
\end{proposition}

Thus, to apply Theorem \ref{Thm-Main}, a natural example consists in choosing 
\begin{equation}
	\label{Def-B1-B2-Ex1}
	B_1 u_1 =  1_\mathcal{O} \left( \begin{array}{c} u_1 \\ 0\end{array} \right), \quad U_1 = L^2(\mathcal{O}), 
	\quad \hbox{ and } \quad 
	B_2 u_2 =  1_\mathcal{O} \left( \begin{array}{c} 0 \\ u_2 \end{array} \right), \quad U_2 = L^2(\mathcal{O}).
\end{equation}
Theorem \ref{Thm-Main} then readily {implies}: 
\begin{theorem}
	\label{Thm-Ex-1}
	 System \eqref{ParaSys-Ex1} is null-controllable in arbitrary small times with controls $u_1$ and $u_2$ in $L^2(0,T; L^2(\mathcal{O}))$ satisfying the additional switching constraints \eqref{Def-Switching-Controls}.
\end{theorem}

\begin{remark}
{By a shifting argument, Theorem \ref{Thm-Ex-1} remains true if we only consider $P$ a bounded symmetric matrix. }
\end{remark}

Here, we would like to emphasize that our results are different from the ones in which the controls may act only on one component. Indeed, in such case, it is clear that more conditions are needed, since when $P = 0$ and controlling on only one component, the second component will be free of control. 

Of course, when $P = 0$, it is easy to check that one can control system \eqref{ParaSys-Ex1} with controls having a switching structure, since one can control the first component $y_1$ to $0$ at time $T/2$ by keeping the control $u_2 = 0$ in $(0,T/2)$ and then control the second component $y_2$ to $0$ on $(T/2,T)$ by keeping the control $u_1 = 0$ in $(T/2, T)$. However, when $P \neq 0$, this strategy does not seem to be applicable directly.

On the other hand, when one wants to control a system through one component only, it is clear that the coupling terms should play an important role, see for instance \cite{DuprezLissy}.

Therefore, our results really comes in between the questions of controllability of parabolic systems when the controls act on all the components of the state and when the controls may act only on one (or some) component of the state.

\subsection{Distributed controls of $3$-d Stokes equations.}
\label{Subsec-distrib-3-d-Stokes}
Let $\Omega$ be a smooth bounded domain of $\R^3$ and let us consider the following Stokes equation: 
\begin{equation}
	\label{Eq-Stokes-3-d}
	\left\{
		\begin{array}{ll}
			\partial_t y - \Delta y + \nabla p = 1_\mathcal{O} u & \hbox{ in } (0,T) \times \Omega, 
			\\
			\div y  = 0 & \hbox{ in } (0,T) \times \Omega, 
			\\
			y = 0 & \hbox{ on } (0,T) \times \partial \Omega, 
			\\
			y(0, \cdot) = y_0 & \hbox{ in } \Omega.
		\end{array}
	\right.
\end{equation}
Here, $y = y(t,x) \in \R^3$ denotes the velocity field of an incompressible fluid, $p$ is the pressure, and the control $u$ acts through the non-empty open subset $\mathcal{O}$ of $\Omega$.

This example fits the setting of Theorem \ref{Thm-Main-Gal} by choosing the state space 
\begin{equation}
	\label{Def-V-0-n}
	H = V^0_n(\Omega) = \{y \in L^2(\Omega; \R^3), \, \div y = 0 \hbox{ in } \Omega \hbox{ and } y\cdot {\bf n}_x = 0 \hbox{ on } \partial \Omega\},
\end{equation}
the operator $A$ as
\begin{equation}
	\label{Stokes-Op}
	A = - \mathbb{P}  \Delta, \hbox{ with } \mathscr{D}(A) = \{ y \in H^2 \cap H^1_0(\Omega; \R^3), \, \div y = 0 \hbox{ in } \Omega\} \hbox{ in } H, 
\end{equation}
where ${\bf n}_x$ is the outward normal to $x \in \partial \Omega$, $\mathbb{P}$ is the orthogonal projection on $V^0_n(\Omega)$ in $L^2(\Omega; \R^3)$, 
and the control operator 
$$
	B u = 1_\mathcal{O} \left( \begin{array}{c} u_1 \\ u_2 \\ u_3 \end{array} \right), \quad \hbox{ with } U = (L^2(\mathcal{O}))^3.
$$
It is then natural to define the operators $B_1$, $B_2$ and $B_3$ as follows: 
\begin{multline}
	\label{Choice-B-Stokes-3d}
	B_1 u_1 = 1_\mathcal{O} \left( \begin{array}{c} u_1 \\ 0 \\ 0 \end{array} \right), 
	\quad
	B_2 u_2 = 1_\mathcal{O} \left( \begin{array}{c} 0 \\ u_2 \\ 0 \end{array} \right), 
	\quad
	B_3 u_3 = 1_\mathcal{O} \left( \begin{array}{c} 0 \\ 0 \\ u_3 \end{array} \right), 
	\\
	\hbox{ with } U_1 = U_2 = U_3 = L^2(\mathcal{O}).
\end{multline}
Indeed, we have the following results:
\begin{itemize}
	\item The operator $A$ is self-adjoint on $V^0_n(\Omega)$, see \emph{e.g.} \cite[Lemma IV.5.4]{Boyer-Fabrie-Book}; 
	\item	The Stokes problem \eqref{Eq-Stokes-3-d} is null-controllable in arbitrary small times, see \cite{Imanuvilov2001}; 
%	\item For each $\ell \in \{1, 2, 3\}$, the Stokes problem \eqref{Eq-Stokes-3-d} is null-controllable in arbitrary small times with controls $u \in L^2(0,T; (L^2(\omega))^3)$ such that $u_\ell$ is identically zero, see \cite{Coron-Guerrero-09}; 
\end{itemize}
We can therefore readily apply Theorem \ref{Thm-Main-Gal}: 
\begin{theorem}
	\label{Thm-Stokes-3D}
	 Given any $y_0 \in V^0_n(\Omega)$, {there exist control functions $u_1$, $u_2$ and $u_3$ in $L^2(0,T; L^2(\mathcal{O}))$} such that the controlled trajectory $y$ of \eqref{Eq-Stokes-3-d} satisfies $y(T ) = 0$ in $\Omega$ and with control functions $u_1$, $u_2$ and $u_3$ satisfying condition \eqref{Def-Switching-3}, \emph{i.e.} such that almost everywhere in $(0,T)$, at most one of the controls $u_1(t)$, $u_2(t)$, $u_3(t)$ is non-zero.
\end{theorem} 
It is interesting to consider this case, since the controllability of the Stokes equation \eqref{Eq-Stokes-3-d} with controls having one or two vanishing components has been studied in the literature. In particular, it has been shown in \cite{Coron-Guerrero-09} that, given $\ell \in \{1, 2, 3\}$, system \eqref{Eq-Stokes-3-d} is null-controllable in arbitrary small times with controls $u \in L^2(0,T; (L^2(\mathcal{O}))^3)$ satisfying $u_\ell \equiv 0$. Besides, the result in \cite{Lions-Zuazua-96} shows that system \eqref{Eq-Stokes-3-d} may be not null-controllable (in fact, not even approximate controllable) in some specific geometric settings with controls having two vanishing components.

Note that the result in \cite{Coron-Lissy} about the null-controllability of the $3$-dimensional incompressible Navier-Stokes equation with controls having two vanishing components strongly uses the non-linear term in the Navier-Stokes equation in the spirit of the celebrated \emph{Coron's return method}, and thus does not apply to the linear problem \eqref{Eq-Stokes-3-d}.

\subsection{Boundary control of a system of coupled heat equations.}

This example is closely related to the one in Section \ref{Subsec-Ex-1}. Let us consider $\Omega$ a smooth bounded domain and  the following parabolic system: 
\begin{equation}
	\label{ParaSys-Ex3}
	\left\{
		\begin{array}{ll}
			\ds \partial_t y - D \Delta y + P y = 0 , & \hbox{ in } (0,T) \times \Omega, 
			\\
			y = u 1_\Gamma & \hbox{ on } (0,T) \times \partial \Omega, 
			\\
			y(0, \cdot) = y_0 & \hbox{ in } \Omega, 
		\end{array}
	\right.
\end{equation}
where 
$$
	y =  \left( \begin{array}{c} y_1 \\ y_2 \\ \vdots \\ y_n  \end{array} \right), 
	\quad
	D = \hbox{diag\, } (d_1, \cdots, d_n), \hbox{ with } d_i >0 \hbox{ for all } i \in \{1, \cdots, n\},
$$
and $P = P(x) \in L^\infty(\Omega; S_{n}^+(\R))$, where $ S_{n}^+(\R)$ denotes the set of symmetric positive definite $n\times n$ matrices with real coefficients. Here, the control 
$$
	u =  \left( \begin{array}{c}u_1  \\ \vdots \\ u_n\end{array} \right), 
$$
acts on the system \eqref{ParaSys-Ex3} on a non-empty open subset $\Gamma$ of the boundary $\partial\Omega$ through the multiplication by the indicator function $1_\Gamma$.

System \eqref{ParaSys-Ex3} fits into the framework of Theorem \ref{Thm-Main-Gal} by setting 
\begin{equation}
	\label{OpA-Ex3}
	A =  - D \Delta_x + P, \quad \hbox{ in } H = (L^2(\Omega))^n \hbox{ with domain } \mathscr{D}(A) = (H^2 \cap H^1_0(\Omega))^n,
\end{equation}
and {the control operator $B$ as follows}:
$$
	B u =  \tilde A \hbox{Dir}_\Gamma(u) \  \hbox{ for } u =  \left( \begin{array}{c} u_1  \\ \vdots \\ u_n \end{array} \right), \quad U = {(L^2(\Gamma))^n}, 
$$
where $ \hbox{Dir}_\Gamma : (L^2(\Gamma))^n \mapsto (L^2(\Omega))^n$ is the Dirichlet operator given by 
$$
	 \hbox{Dir}_\Gamma u = z, \, \hbox{ where $z$ solves }
	\left\{
		\begin{array}{ll}
			\ds  - D \Delta z + P z = 0 , & \hbox{ in }  \Omega, 
			\\
			\ds z = u 1_\Gamma & \hbox{ on } \partial \Omega, 
		\end{array}
	\right. 
$$
and $\tilde A$ denotes the extension of $A$ of domain $(L^2(\Omega))^n$ on $((H^2 \cap H^1_0(\Omega))^n)'$, see \cite[Proposition 3.4.5 and Section 10.7]{TWBook}.

Similarly as in Proposition \ref{Prop-Ex1-Cont}, one can show using classical Carleman estimates (see \cite{FursikovImanuvilov}) that:
\begin{proposition}
	\label{Prop-Ex3-Cont}
	System \eqref{ParaSys-Ex3} is null-controllable in arbitrarily small times with control functions $u = (u_1, \cdots, u_n)$ in  $ L^2(0,T; (L^2(\Gamma))^n)$.
\end{proposition}

One can then readily apply Theorem \ref{Thm-Main-Gal}:
\begin{theorem}
	\label{Thm-Ex-3}
	 System \eqref{ParaSys-Ex3} is null-controllable in arbitrary small times with controls $u = (u_1, \cdots, u_n)$ in  $ L^2(0,T; (L^2(\Gamma))^n)$ satisfying the additional switching constraints \eqref{Def-Switching-3}.
\end{theorem}
Again, we emphasize that our results really complements the ones in which the controls act only on one component of the system, in which the situation is much more intricate since controllability results will depend on delicate coupling conditions, see for instance \cite{Ammar-Khodja-et-al-2016} and references therein. 

\subsection{Boundary control of $3$-d Stokes equations.}

Again, one can also consider Stokes equations, but this time controlled from the boundary. Using \cite{Imanuvilov2001} (see also \cite{FerCarGueImaPuel}), in a smooth bounded domain $\Omega \subset \R^3$, the $3$-d Stokes equations is null-controllable in any time $T$ through any non-empty open subset of its boundary. To be more precise, we let $\Omega$ be a smooth bounded domain of $\R^3$ and $\Gamma$ a non-empty open subset of $\partial \Omega$, and we consider the following Stokes equation: 
\begin{equation}
	\label{Eq-Stokes-3-d-Bound}
	\left\{
		\begin{array}{ll}
			\partial_t y - \Delta y + \nabla p = 0 & \hbox{ in } (0,T) \times \Omega, 
			\\
			\div y  = 0 & \hbox{ in } (0,T) \times \Omega, 
			\\
			y = 1_\Gamma(x) u & \hbox{ on } (0,T) \times \partial \Omega, 
			\\
			y(0, \cdot) = y_0 & \hbox{ in } \Omega,
		\end{array}
	\right.
\end{equation}
where $1_\Gamma$ is the indicator function of the set $\Gamma$, and $u$ is assumed to belong to $L^2(0,T; L^2(\Gamma; \R^3))$ and satisfy 
\begin{equation}
	\label{Incompress}
	\forall t \in (0,T), \quad \int_\Gamma u(t,x ) \cdot {\bf n}_x \, d\sigma = 0 
\end{equation}
where ${\bf n}_x$ is the outward normal to $\partial \Omega$ at $x \in \partial \Omega$. {Condition \eqref{Incompress} can be seen as a compatibility condition} with the divergence free condition $\div y = 0$ and can be obtained immediately by integrating it in $\Omega$.

Properly speaking, \cite{Imanuvilov2001} does not deal with boundary controls, but the following result can be easily obtained from \cite{Imanuvilov2001} using the classical extension/restriction argument to get controllability results with controls on the boundary:
\begin{theorem}[\cite{Imanuvilov2001}]
	{System \eqref{Eq-Stokes-3-d-Bound}} is null controllable in any time $T$. To be more precise, for all $T>0$, for any $y_0 \in V^0_n(\Omega)$, there exists a control function $u \in L^2(0,T; L^2(\Gamma; \R^3))$ satisfying \eqref{Incompress} such that the controlled {trajectory $y$ of \eqref{Eq-Stokes-3-d-Bound}} satisfies $y(T) = 0$ in $\Omega$.
\end{theorem}
%
%In fact, we will be using the following results, with smoother controls, which is proved in \cite[Proposition 4]{Coron-Lissy} (it can also be easily deduced from the spectral estimates in \cite{Chaves-Lebeau} combined with the Lebeau-Robbiano strategy \cite{LebRob}), using again the classical extension/restriction argument to get controllability results with controls on the boundary:
%%
%\begin{theorem}
%	System \eqref{Eq-Stokes-3-d} is null controllable in any time $T$. To be more precise, for all $T>0$, for any $y_0 \in V^0_n(\Omega)$, for all $m \in \N$, there exists a control function $u \in H^m_0(0,T; L^2(\Gamma; \R^3)) \cap L^2(0,T; H^{2m}_0(\Gamma; \R^3))$ satisfying \eqref{Incompress} such that the controlled trajectory $y$ of \eqref{Eq-Stokes-3-d} satisfies $y(T) = 0$ in $\Omega$.
%\end{theorem}	
%
Because of condition \eqref{Incompress}, it is natural to decompose the space $\{ u \in L^2(\Gamma; \R^3): \ \int_\Gamma u(x ) \cdot {\bf n}_x \, d\sigma = 0 \}$ using tangential and normal components of $u$. Therefore, we choose a family of triplets $(e_1(x), e_2(x), {\bf n}_x)$ indexed by $x \in \Gamma$ such that for all $x \in \Gamma$, $(e_1(x), e_2(x), {\bf n}_x)$ is an orthonormal basis of $\R^3$, and we define $U_1 = U_2 = L^2(\Gamma; \R)$ and $U_3 = \{ u_3 \in L^2(\Gamma; \R) \hbox{ with } \int_\Gamma u_3 (x) d \sigma = 0\}$, also denoted by $ L^2_0(\Gamma; \R)$, and the isomorphism $\pi$ in \eqref{Link-U1-2-3-U} is then given by 
\begin{equation}
	\label{isomorphisms-Stokes}
	\pi : (u_1, u_2, u_3) \in U_1 \times U_2 \times U_3 \mapsto \left( x \mapsto (u_1(x) e_1(x) + u_2 (x) e_2(x) + u_3 (x) {\bf n}_x)\right).
\end{equation}
Now, as before, see e.g. \cite{Raymond-2007}, to properly define the operator $B$ in this case, we need to introduce the Dirichlet operator $D_\Gamma$ defined by 
$$
	D_\Gamma u = z, \hbox{ where $z$ solves }
	 \left\{
		\begin{array}{ll}
			-\Delta z + \nabla p = 0 &\hbox{ in } \Omega, 
			\\
			\hbox{div\,} z = 0 &\hbox{ in } \Omega, 
			\\
			z = 1_\Gamma u & \hbox{ on } \partial \Omega, 
		\end{array}
	 \right.
$$
and the operator $B$ is defined by 
$$
	B u = \tilde A \mathbb{P} D_\Gamma u, 
$$
where $\tilde A$ denotes the extension of the Stokes operator (defined in \eqref{Def-V-0-n}--\eqref{Stokes-Op}) from $V^0_n(\Omega)$ to $\mathscr{D}(A)'$ and $\mathbb{P}$ denotes the Leray projection, that is the orthogonal projection on $V^0_n(\Omega)$ in $L^2(\Omega; \R^3)$, . The full system \eqref{Eq-Stokes-3-d-Bound} can then be written as 
\begin{equation}
	\label{Eq-Stokes-Abstract}
	\left\{
		\begin{array}{ll}
			\mathbb{P}y' + \tilde A \mathbb{P}y = B u, \quad & t \in (0,T), 
			\\ 
			\mathbb{P}y(0) = \mathbb{P}y_0, 
			\\
			(I - \mathbb{P}) y = (I - \mathbb{P}) D_\Gamma u, \quad &t \in (0,T).
		\end{array} 
	\right.
\end{equation}
Accordingly, the quantities $\mathbb{P} y$ and $( I - \mathbb{P}) y$ should be handled separately. In particular, see \cite[Theorem 2.3 and Theorem 3.1]{Raymond-2007} for $u \in L^2(0,T; L^2(\Gamma; \R^3))$ satisfying \eqref{Incompress}, the solution $y$ of \eqref{Eq-Stokes-Abstract} with initial datum $\mathbb{P} y_0 \in V^0_n(\Omega)$ satisfies $\mathbb{P} y \in L^2(0,T; V^0_n(\Omega)) \cap_{\varepsilon>0} L^2(0,T; V^{1/2- \varepsilon}(\Omega)) \cap H^{1/4}(0,T; V^0( \Omega) )\cap C^0([0,T]; V^{-1}(\Omega))$ and $(I- \mathbb{P}) y \in L^2(0,T; V^{1/2}(\Omega))$. Here, $V^0_n(\Omega)$ is the space defined in \eqref{Def-V-0-n}, and the other spaces are 
\begin{align*}
	& V^{s}(\Omega) = \{ y \in H^s(\Omega; \R^3), \, \text{ div \,} y = 0 \text{ in } \Omega, \text{ with } \langle y \cdot n , 1 \rangle_{H^{-1/2}(\partial \Omega), H^{1/2}(\partial \Omega)} = 0 \}, \qquad (s \geq 0),
	\\
	& V^1_0(\Omega) = \{ y \in H^1_0(\Omega; \R^3), \, \text{ div \,} y = 0 \text{ in } \Omega \} , 
\end{align*}
and $V^{-1}(\Omega)$ is the dual of $V^1_0(\Omega)$ with $V^0_n(\Omega)$ as pivot space. 

Theorem \ref{Thm-Main-Gal} then yields the following result: 
\begin{theorem}
	\label{Thm-Stokes-3d-Switch}
	Let $\Omega$ be a smooth bounded domain of $\R^3$, $\Gamma$ a non-empty open subset of $\partial \Omega$. Given a family of orthonormal triplets $(e_1(x), e_2(x), {\bf n}_x)$ for $x \in \Gamma$ which defines the control operators $B_1$, $B_2$ and $B_3$ according to \eqref{Def-B1-B2-B3} through the isomorphism $\pi $ in \eqref{isomorphisms-Stokes}, the control system \eqref{Eq-Stokes-3-d-Bound} is null-controllable in arbitrary small times with controls $(u_1, u_2, u_3) \in L^2(0,T; L^2(\Gamma; \R)^2 \times  L^2_0(\Gamma; \R))$ satisfying the switching condition \eqref{Def-Switching-3} in the following sense: for any $T>0$, for any $y_0 \in V^0_n(\Omega)$, {there exist control functions $u_1$, $u_2$ in  $L^2(0,T; L^2(\Gamma; \R))$, and  $u_3 \in L^2(0,T; L^2_0(\Gamma; \R))$} satisfying the switching condition \eqref{Def-Switching-3} such that the solution $ y$ of \eqref{Eq-Stokes-Abstract} satisfies $\mathbb{P} y(T) = 0$. 
\end{theorem}

\begin{remark}
	Although Theorem \ref{Thm-Stokes-3d-Switch} states only the control of $\mathbb{P} y$ at time $T$, extending the controls $(u_1, u_2, u_3)$ by $0$ for $t \geq T$, one easily checks that $\mathbb{P}y$ and $(I - \mathbb{P} ) y$ vanishes for $t \geq T$. The difficulty is that $(I - \mathbb{P}) y$ does not a priori make sense at time $T$ since it only belongs to $L^2(0,T; V^{1/2}(\Omega))$.
\end{remark}

It is to be noted that there are, up to our knowledge, almost no result regarding the controllability of Stokes system with controls acting only on normal or tangential components only. We are only aware {of \cite{Fursikov-1995} when considering tangential controls on the whole boundary} and of  the results by \cite{ChowdhuryMitraRenardy-2017} for the Stokes equation in a channel when the control is localized on the whole boundary of one side of the channel.

\section{Extensions}
\label{Sec-Extensions}

Theorem \ref{Thm-Main-Gal} focuses on the case of operators $A$ which are either positive self-adjoint with compact resolvent or are matrices. It is thus natural to also consider the case of general operators $A$ which generates an analytic semi-group and are possibly non self-adjoint. The goal of this section is precisely to discuss this case. As we will see, our arguments will require the introduction of several spectral assumptions which are hard to check in practice. 

\begin{theorem}
	\label{Thm-Ext}
	Let $A$ be an operator on the Hilbert space $H$ having compact resolvent and such that $-A$ generates an analytic semigroup. 
	
	Assume that the Hilbert space $H$ can be decomposed as 
	\begin{equation}
		\label{Total-Space}
		H  = \oplus_{k \in \N} H_k, \quad \text{ where } H_k \text{ are finite-dimensional vector spaces} 
	\end{equation}
	such that for all $k \in \N$, 
	\begin{multline}
		\label{A-k-Diag-Nilp}
		A^*(H_k) \subset H_k, \text{ and } 
		A^*|_{H_k} = A_k^*, 
		\\
		\text{ where $A_k^*$ is of the form $ \lambda_k I + N_k$, with $\lambda_k \in \C$ and $N_k$ nilpotent}.
	\end{multline} 
	Also assume for simplicity that $\Re(\lambda_0) \leq \Re(\lambda_1) \leq \cdots \leq \Re(\lambda_k) \leq \cdots \to \infty$. 
	
	Furthermore, denoting $\mathbb{P}_k$ the projection on $H_k$ parallel to $\oplus_{ j \neq k} H_j$, we assume that there exists $t_0 >0$ large enough so that 
	\begin{equation}
		\label{Hyp-Cvgce-Projections-Spectrales}
		\forall t \geq t_0, \quad e^{- t A^*} = \sum_k e^{- t A_k^*} \mathbb{P}_k, 
	\end{equation}
	\emph{i.e.} the right hand side is norm convergent for $t >  t_0$.
	
	Let $B \in \mathscr{L}(U,\mathscr{D}(A^*)')$, where $U$ is an Hilbert space, let $n \in \N$ with $n \geq 2$, and assume that $U$ is isomorphic to $U_1 \times \cdots \times U_n$ for some Hilbert spaces $U_i$, $i \in \{1, \cdots, n\}$, and define $B_i$ for $i \in \{1, \cdots, n\}$ as in \eqref{Def-B1-B2-B3}.

	We assume that system \eqref{Main-Eq-u} is null-controllable in arbitrary small times.

	Then the system \eqref{Main-Eq-u-1-2-3}	is null-controllable in arbitrary small times with switching controls, \emph{i.e.} satisfying \eqref{Def-Switching-3}. To be more precise, given any $T>0$ and any $y_0 \in H$, there exist $n$ control functions $u_i \in L^2(0,T; U_i)$, $i \in \{1, \cdots, n\}$ such that the solution $y$ of \eqref{Main-Eq-u-1-2-3} satisfies \eqref{Null-Cont-Req} while the control functions satisfy the switching condition \eqref{Def-Switching-3}.
\end{theorem}
Before going further and giving the proof of {Theorem \ref{Thm-Ext}}, let us emphasize that the assumptions on $A^*$ may be delicate to prove for general operators $A$ generating an analytic semigroup.

Of course, each $H_k$ corresponds to the generalized eigenspaces corresponding to the eigenvalues $\lambda_k$, and the projections $\mathbb{P}_k$ corresponds to the spectral projections. However, condition \eqref{Hyp-Cvgce-Projections-Spectrales} is difficult to check in practice, see e.g. \cite{Gohberg-Krein-1969} for an introduction to spectral theory for non self-adjoint operators.

To better illustrate that fact, we present two examples of interest. 

The first one is borrowed from \cite{Benabdallah-Boyer-Morancey-2020} and we deeply thank Franck Boyer for having pointed it to us. 

Let us take $A_0$ a positive self-adjoint operator with compact resolvent defined on a Hilbert space $H_0$ with domain $\mathscr{D}(A_0)$, which we will assume for simplicity to have only simple eigenvalues. Then, for $f \in \mathscr{C}^\infty(\R_+^*; \R_+^*)$ bounded at infinity, define 
\begin{equation}
\label{Ex-Boyer}
	\widehat{A} = \left( \begin{array}{cc} A_0 & Id \\ 0 & A_0 + f(A_0) \end{array} \right), \hbox{ in } H = (H_0)^2, \quad \hbox{ with } \mathscr{D}(\widehat{A}) = (\mathscr{D}(A_0))^2.
\end{equation}
It is easy to check that such $\widehat{A}$ generates an analytic semigroup in $H$, since it is a bounded perturbation of the operator $\text{Diag\,} (A_0, A_0)$. Besides, its spectrum can be expressed easily in terms of those of $A_0$. If $(\lambda_{k,0})_{k \in \N}$ is the set of eigenvalues of $A_0$, corresponding to a family of normalized eigenvectors $(\varphi_{k,0})_{k \in\N}$, then  it is easy to check that the eigenvalues of $\widehat{A}$ are given by the family $(\lambda_{k,1}, \lambda_{k,2})_{k \in \N}$ with $\lambda_{k,1} = \lambda_{k,0}$ and $\lambda_{k,2} = \lambda_{k, 0} + f(\lambda_{k,0})$. The corresponding eigenvectors are given for $k \in\N$ by 
\[
	\varphi_{k,1} =  \left( \begin{array}{c}1 \\ 0 \end{array} \right)  \varphi_{k,0} , 
	\qquad
	\varphi_{k,2} =  \frac{1}{\sqrt{1+ f(\lambda_{k,0})^2} }\left( \begin{array}{c} 1 \\ f(\lambda_{k,0})  \end{array} \right)  \varphi_{k,0} . 
\]
It is then easy to check that 
\begin{align*}
	&\mathbb{P}_{k,1} \left( \begin{array}{c}z_1 \\ z_2 \end{array} \right) = \varphi_{k,1} 
	\left\langle
		\left( \begin{array}{c} 1\\  - \frac{1}{f(\lambda_k)} \end{array}  \right) \varphi_{k,0},\left( \begin{array}{c}z_1 \\ z_2 \end{array} \right)
	\right\rangle_{H}, 
	%\left( \mathbb{P}_{k,0} z_1 - \frac{1}{f(\lambda_k)}\mathbb{P}_{k,0} z_2\right)
	\\ 
	& \mathbb{P}_{k,2} \left( \begin{array}{c}z_1 \\ z_2 \end{array} \right) = \varphi_{k,2} 
	\left\langle
		\left( \begin{array}{c} 0\\ \frac{\sqrt{1+f(\lambda_k)^2}}{f(\lambda_k)}\end{array}  \right) \varphi_{k,0},\left( \begin{array}{c}z_1 \\ z_2 \end{array} \right)
	\right\rangle_{H}
	%\left(  \frac{\sqrt{1+f(\lambda_k)^2}}{f(\lambda_k)}\mathbb{P}_{k,0} z_2\right)
	.
\end{align*}
When $f$ goes to zero at infinity, the norms of these projections behave like $1/f(\lambda_k)$. In particular, if for $T_0 >0$, there exists $C$ such that $f(s) \leq C e^{ - T_0s}$ for $s$ large enough, we see that the right hand side of \eqref{Hyp-Cvgce-Projections-Spectrales} is not norm convergent for $t \in (0,T_0)$. Of course, this also means that when considering $f(s) = \exp(-s^2)$, condition \eqref{Hyp-Cvgce-Projections-Spectrales} is not satisfied whatever $t_0 >0$ is.

This example shows that even for rather gentle perturbations of self-adjoint operators, condition \eqref{Hyp-Cvgce-Projections-Spectrales} should be analyzed with caution. 

We also present another example in this direction, based on the works \cite{Davies-2000, Davies-et-al-04}  discussing the operator $A_\alpha$ defined for complex number $\alpha \in \C \setminus \{0\}$ with $\text{Arg\,} (\alpha) < \pi /4$ on $L^2(\R)$ by 
\[
	A_\alpha y = -  \alpha^{-2} y'' + \alpha^2 x^2 y.
\]
In fact, to be perfectly rigorous, the operator $ A_\alpha$ has to be defined as the closed densely defined operator associated to the quadratic form 
\[
	\int_\R \left(\alpha^{-2} |y' (x)|^2 + \alpha^2 x^2 y(x)^2 \right) dx, 
\]
originally defined on $\mathscr{C}^\infty_c(\R)$. 

According to \cite{Davies-2000}, the eigenvalues of the operator $A_\alpha$ does not depend on $\alpha$ for $\alpha \in \C \setminus \{0\}$ with $|\text{Arg\,} (\alpha) | < \pi /4$ and thus coincides with the usual ones for the harmonic operator (which are $2 \N + 1$), but except if $\alpha \in \R_+^*$, the spectrum of $A_\alpha$ is wild (\cite[Theorem 9]{Davies-2000}), meaning that, denoting by $\mathbb{P}_k$ the spectral projector on the $k$-th eigenvector,  $\| \mathbb{P}_k\|$ cannot be bounded by a polynomial in $k$. 

In fact, the situation is even worse and, for $\alpha \notin \R$, the formula 
\[
	e^{ - t \alpha^2 A_\alpha} = \sum_{k \in \N} e^{- t \alpha^2 \lambda_k} \mathbb{P}_k, 
\]
holds only for $t$ large enough, see \cite[Corollary 4]{Davies-et-al-04}, due to the fact that $\norm{\mathbb{P}_k}$ behaves like $\exp(c \Re( \lambda_k))$ for some strictly positive $c$ as $k \to \infty$.

To sum up, we see that condition \eqref{Hyp-Cvgce-Projections-Spectrales} is rather delicate to deal with. Although it is automatically satisfied in finite dimensional contexts or when $A$ is self-adjoint, when considering general operators $A$ generating an analytic semigroup, condition \eqref{Hyp-Cvgce-Projections-Spectrales} should be carefully analyzed.

\begin{proof}
	The proof of Theorem \ref{Thm-Ext} strongly follows the ones of Theorem \ref{Thm-Main} and \ref{Thm-Main-Gal}. 

	For sake of simplicity, we will only focus on the case $n = 2$ and $B \in \mathscr{L}(U,H)$, similarly as in Theorem \ref{Thm-Main}, since the general case $n \geq 3$ and {$B \in \mathscr{L}(U, \mathscr{D}(A^*)')$} can be handled similarly as in Section \ref{Sec-Thm-Gal} by minor adaptations of the case $n = 2$. 

	In fact, it is easy to check that the only point which needs further analysis is the counterpart of Lemma \ref{Lemma-A-SelfAdj} and Lemma \ref{Lemma-H-Finite-Dim}. 	
	
	We thus take $X$ as in \eqref{Def-X-Space}, and let $Z_T \in X$ be a minimizer of the functional $J$ in \eqref{Def-J-H}, and we study the set $I$ defined in \eqref{Def-I-Switching-Times}.
	\begin{lemma}
		\label{Lem-Ext}
		Assume that $A$ is an operator on the Hilbert space $H$ having compact resolvent and such that $-A$ generates an analytic semigroup. Also assume that the Hilbert space $H$ can be decomposed as in \eqref{Total-Space} such that $A^*$ satisfies \eqref{A-k-Diag-Nilp} for all $k \in \N$, where the corresponding eigenvalues $(\lambda_k)_{k \in \N}$ are ordered  
such that $\Re(\lambda_0) \leq \Re(\lambda_1) \leq \cdots \leq \Re(\lambda_k) \leq \cdots \to \infty$. Furthermore assume that, denoting $\mathbb{P}_k$ the projection on $H_k$ parallel to $\oplus_{ j \neq k} H_j$, there exists $t_0 >0$ large enough such that \eqref{Hyp-Cvgce-Projections-Spectrales} holds.

		Define the set $W$ as in \eqref{Def-W}.

		Let $B \in \mathscr{L}(U, H)$ and assume that system \eqref{Main-Eq-u} is null-controllable in arbitrary small times.

		Then, for $\alpha$ as in \eqref{Example-Alpha} with $\omega \in \R \setminus W$, the set $I$ is necessarily of zero measure, except in the trivial case $ \| B_1^* Z \|_{L^2(0,T;U_1)} = \| B_2^* Z \|_{L^2(0,T; U_2)}  =0$.
	\end{lemma}	

	Once Lemma \ref{Lem-Ext} will be proved (see afterwards), the end of the proof of Theorem \ref{Thm-Ext} will follow line to line the one of Theorem \ref{Thm-Main}, by showing that the Euler Lagrange equation satisfied by $Z_T$ is given by \eqref{Euler-Lag-Eq} when $Z_T \neq 0$, entailing that the controls $u_1$ and $u_2$ given by \eqref{Def-Controls} are of switching forms and indeed control the equation \eqref{Main-Eq-u-1-2}. As before, the case $Z_T = 0$ corresponds to the case $y_0 = 0$, and then taking the controls $u_1$ and $u_2$ to be identically zero solves the problem.
\end{proof}
	
\begin{proof}[Proof of Lemma \ref{Lem-Ext}]
	In order to prove that the set $I$ is of zero measure except when $ \| B_1^* Z \|_{L^2(0,T;U_1)} = \| B_2^* Z \|_{L^2(0,T; U_2)}  =0$, we consider a strictly positive and strictly increasing sequence $T_n$ going to $T$ as $n \to \infty$ and we show that for all $n \in \N$, the set $I_n = I \cap (0, T_n)$ is of zero measure except in the trivial case in which both $B_1^* Z$ and $B_2^* Z$ vanish identically on $(0,T_n)$. 
	
	As in the proof of Lemma \ref{Lemma-A-SelfAdj}, the small time null-controllability implies that since $Z_T \in X$, the trajectory $Z|_{(0,T_n)}$ is well defined and in fact solves the equation \eqref{Eq-Z-n} with some initial datum $Z_n \in H$. 
	
	Accordingly, since $-A^*$ generates an analytic semigroup, the function $t \mapsto Z(t)$ is in fact analytic on $(0,T_n)$ with values in $H$, and can be extended analytically to $(-\infty, T_n)$. 
	
	We now assume that $I_n$ is not of zero measure. According to the analyticity properties above, this implies that the identity \eqref{Infinite-Time-Identity} holds. 
	
	To conclude as in the proof of Lemma \ref{Lemma-A-SelfAdj} or Lemma \ref{Lemma-H-Finite-Dim}, we would like to write formula \eqref{Exp-Z(t)}. This cannot be done for all $t< T_n$ as before, but according to \eqref{Hyp-Cvgce-Projections-Spectrales}, it is still true for $t < T_n -T_0$:
	\begin{equation}
	\label{Exp-Z(t)-Mod}
		\forall t < T_n - T_0, \quad Z(t) = \sum_{k \in \N} e^{A_k^* {(t - T_n})} \mathbb{P}_k Z_n.
	\end{equation} 

	Each $H_k$ is a finite-dimensional vector space. Therefore, writing the Jordan decomposition of $A^*|_{H_k}$, for each $k \in \N$, denoting by $m_k$ the size of the maximal Jordan block corresponding to $\lambda_k$,	%
	\[
		e^{A_k^* {(t - T_n)}} = e^{\lambda_k(t-T_n)} \sum_{\ell \in \{0, \cdots, m_k\} } \frac{(t-T_n)^\ell}{\ell !} N_k^\ell.
	\]

	We then follow the proof of Lemma \ref{Lemma-H-Finite-Dim}, introducing 
	$$
		k_0 = \inf \left\{ k \in \N : \exists \ell \in \{0, \cdots, m_k\} \hbox{ such that }    \| B_1^* N_k^\ell \mathbb{P}_{k} Z_n \|_{U_1}+ \| B_2^* N_k^\ell \mathbb{P}_{k} Z_n \|_{U_2} \neq 0 \right\}.
	$$
	Our goal is to show that $k_0$ is necessarily infinite. Indeed, if $k_0$ is infinite, then for all $k$, $B_1^* e^{A_k^*(t-T_n)} \mathbb{P}_k Z_n$ and $B_2^* e^{A_k^*(t-T_n)} \mathbb{P}_k Z_n$ identically vanish, so that using the formula \eqref{Exp-Z(t)-Mod}, we see that $B_1^* Z$ and $B_2^* Z$ identically vanish on $(-\infty, T_n - T_0)$, and by analyticity on $(0, T_n)$ as well.
	
	We prove that $k_0$ is necessarily infinite by contradiction, assuming that $k_0$ is finite. 
	
	Next, we define $\ell_1$ by
	\begin{align*}
		\ell_1 = \sup \left\{\ell: \exists k \hbox{ with } \Re(\lambda_{k}) = \Re(\lambda_{k_0}) \  \hbox{ and }  \| B_1^* N_k^\ell \mathbb{P}_{k} Z_n \|_{U_1}+ \| B_2^* N_k^\ell \mathbb{P}_{k} Z_n \|_{U_2} \neq 0 \right\}, 
	\end{align*}
	and the set 
	$$
		D = \left \{ k :  \Re(\lambda_{k}) = \Re(\lambda_{k_0}) \hbox{ and } \| B_1^* N_k^{\ell_1} \mathbb{P}_{k} Z_n \|_{U_1}+ \| B_2^* N_k^{\ell_1} \mathbb{P}_{k} Z_n  \|_{U_2} \neq 0 \right \}. 
	$$

	According to the above definition, we can decompose $Z$ as 
	\begin{align*}
		Z_d(t) & =  e^{\Re(\lambda_{k_0})t} \frac{(T - t)^{\ell_1} }{\ell_1 !}\sum_{k \in D} N_k^{\ell_1} \mathbb{P}_k Z_n e^{i \Im(\lambda_k) (t-T)},  \qquad  &(t \in (-\infty, T_n)),
		\\
		Z_{d, 2}(t) & =  \sum_{k \text{ with } \Re(\lambda_k) = \Re(\lambda_{k_0}) } e^{\lambda_k(t - T_n)} \left( \sum_{\ell \in \{0, \cdots, \ell_1 - 1 \} } \frac{(T-t)^\ell}{\ell !} N_k^\ell \mathbb{P}_k Z_n  \right) , \qquad &(t \in (-\infty, T_n)),
		\\
		Z_{d,3}(t) & =   \sum_{k \text{ with } \Re(\lambda_k) = \Re(\lambda_{k_0}) } e^{\lambda_k(t - T_n)} \left( \sum_{\ell \geq \ell_1 + 1 } \frac{(T-t)^\ell}{\ell !} N_k^\ell\mathbb{P}_k Z_n  \right) , \qquad &(t \in (-\infty, T_n)),
		\\
		Z_0(t) & = \sum_{k \text{ with } \Re(\lambda_k) < \Re(\lambda_{k_0} )}
			e^{A_k^* (t-T_n)} \mathbb{P}_k Z_n, \qquad &(t \in (-\infty, T_n)), 
		\\
		Z_r(t) & = \sum_{ k \text{ with } \Re( \lambda_k) > \Re(\lambda_{k_0})} e^{A_k^* (t-T_n)} \mathbb{P}_k Z_n, \qquad & (t \in (-\infty, T_n - T_0)). 
	\end{align*}
	By definition of $k_0$ and $\ell_1$, we easily see that
	\begin{equation}
		\label{Zeros-Obs}
		\forall t \in (- \infty, T_n),
		\quad \norm{B_1^* Z_{d,3}(t)}_{U_1} + \norm{B_1^* Z_0(t)}_{U_1} + \norm{B_2^* Z_{d,3}(t)}_{U_2} + \norm{B_2^* Z_0(t)}_{U_2} = 0.
	\end{equation}
	
	It is also easy to check, since the sum defining $Z_{d,2}$ is finite, that there exists a constant $C$ such that $Z_{d,2}$ satisfies
	\begin{equation}
		\label{Smaller-Obs-1}
		\forall t \leq {T_n-1}, \quad \norm{B_1^* Z_{d,2}(t)}_{U_1} + \norm{B_2^* Z_{d,2}(t)}_{U_2} \leq {e^{\Re(\lambda_{k_0})t}} C (T_n-t)^{\ell-1}.
	\end{equation}

	We claim that there exist constants $C$ and $\mu > \Re(\lambda_{k_0})$ such that 
	\begin{equation}
		\label{Decay-Z-r}
		{\forall t \leq T_n - T_0-1}, \quad\norm{Z_r(t) }_H \leq C e^{\mu t }.
	\end{equation}
	Indeed, denoting $A_r^* = A^*|_{\oplus_{k \text{ with } \Re(\lambda_k) > \Re(\lambda_{k_0})} H_k}$, $Z_r$ solves 
	\[
		- Z_r' + A_r^* Z_r = 0, \quad t \in (- \infty, T_n - T_0), \qquad
		 Z_r|_{t = T_n - T_0} = \sum_{ k \text{ with } \Re( \lambda_k) > \Re(\lambda_{k_0})} e^{- A_k^* T_0} \mathbb{P}_k Z_n. 
	\]
	Since $A^*$ generates an analytic semigroup on $H$, it is easy to check that $A_r^* = A^*|_{\oplus_{k \text{ with } \Re(\lambda_k) > \Re(\lambda_{k_0})} H_k}$ also generates an analytic semigroup on ${\oplus_{k \text{ with } \Re(\lambda_k) > \Re(\lambda_{k_0})} H_k}$ and that its spectral abcissa is given by  $\inf \{ \Re(\lambda_k), \text{ with } \Re(\lambda_k) > \Re(\lambda_{k_0})\}$. According to \cite[Theorem 4.3]{Pazy}, $Z_r$ thus decays exponentially at any rate smaller than \[ \inf \{ \Re(\lambda_k), \text{ with } \Re(\lambda_k) > \Re(\lambda_{k_0})\}.\] Since this quantity is strictly larger than $\Re(\lambda_{k_0})$, we have proved \eqref{Decay-Z-r}. 

	Estimate \eqref{Decay-Z-r} in turns imply that 
	\begin{equation}
		\label{Smaller-Obs-2}
		\forall t \leq T_n - T_0, \quad \norm{B_1^* Z_{r}(t)}_{U_1} + \norm{B_2^* Z_{r}(t)}_{U_2} \leq C e^{\mu t}, 
	\end{equation}
	for some $\mu > {\Re(\lambda_{k_0})}$. 

	Using then the identity \eqref{Infinite-Time-Identity}, and the decay estimates \eqref{Zeros-Obs}, \eqref{Smaller-Obs-1} and \eqref{Smaller-Obs-2}, we easily obtain the counterpart of \eqref{Est-Z-d}, that is the existence of positive constants $C_1, C_2$ such that for all {$t \leq T_n - T_0-1$}, 
	\begin{multline}
		\left|
		\norm{B_1^* \left(\sum_{k \in D} N_k^{\ell_1} \mathbb{P}_k Z_n e^{i \Im(\lambda_k) (t-T)}\right)}_{U_1}^2 
		- \alpha(t) 
		\norm{B_2^* \left(\sum_{k \in D} N_k^{\ell_1} \mathbb{P}_k Z_n e^{i \Im(\lambda_k) (t-T)}\right)}_{U_2}^2 
		\right|
		\\
		\leq \frac{C}{T_n - t}.
	\end{multline}
	As in the proof of Lemma \ref{Lemma-H-Finite-Dim}, we then easily get that, if $\alpha$ is as in \eqref{Example-Alpha} with $\omega \notin W$, for all $k \in D$, 
	\[
	\norm{B_1^*  N_k^{\ell_1} \mathbb{P}_k Z_n }_{U_1} + \norm{B_2^*  N_k^{\ell_1} \mathbb{P}_k Z_n }_{U_2} = 0.
	\]
	This contradicts the definition of $k_0$ { when $k_0<\infty$,} and concludes the proof of Lemma \ref{Lem-Ext}.
\end{proof}	
	%
%
%%%%%%%%%%%%%%%%%%%%%%%%%%%%%%%%%%%%%%%%%%%
%
\section{Further comments and open problems}\label{Sec-Further-Com-and-Op-Pbs}
\subsection{Further comments}\label{Subsec-Further}

\paragraph{Approximate controllability.} In this article, we focused on the null-controllability property, but several other notions can be used and developed similarly. For instance, we could consider the approximate controllability property at time $T$, which reads as follows for system \eqref{Main-Eq-u}: For any $y_0 \in H$ and $\varepsilon>0$, there exists $u \in L^2(0,T)$ such that the solution $y$ of \eqref{Main-Eq-u} satisfies $\| y(T) \|_{H} \leq \varepsilon$. 

It is classical, see for instance \cite{Lions-1992}, that this is equivalent to the following unique continuation property for the adjoint equation: if $z_T \in H$ is such that the solution $z$ of \eqref{Adj-Eq} satisfies $B^*z = 0$ in $L^2(0,T;U)$, then $z_T = 0$. 

In this context, following the same strategy as before, we can prove the following counterpart of Theorem \ref{Thm-Main}:
\begin{theorem}
	\label{Thm-Main-App}
	Let us assume one of the two conditions: 
	\begin{itemize}
		\item $A: \mathscr{D}(A) \subset H \to H$ is a self-adjoint positive definite operator with compact resolvent, $H$ being a Hilbert space; 
		\item $H$ is a finite dimensional vector space.
	\end{itemize}
	Let $B \in \mathscr{L}(U,H)$, where $U$ is an Hilbert space, and assume that $U$ is isomorphic to $U_1 \times U_2$ for some Hilbert spaces $U_1$ and $U_2$, and define $B_1$ and $B_2$ as in \eqref{Def-B1-B2}.

	Assume that system \eqref{Main-Eq-u} is approximately controllable at time $T$. 

	Then system \eqref{Main-Eq-u-1-2} is approximately controllable at time $T$ with switching controls, \emph{i.e.} satisfying \eqref{Def-Switching-Controls}. To be more precise, given any $\varepsilon>0$ and any $y_0 \in H$, there exist control functions $u_1 \in L^2(0,T; U_1)$ and $u_2 \in L^2(0,T; U_2)$ such that the solution $y$ of \eqref{Main-Eq-u-1-2} satisfies $\| y(T) \|_{H} \leq \varepsilon$ while the control functions satisfy the switching condition \eqref{Def-Switching-Controls}.
\end{theorem}

The proof of Theorem \ref{Thm-Main-App} can be done exactly similarly as the one of Theorem \ref{Thm-Main}, by minimizing instead of $J$ in \eqref{Def-J-H}, the functional $J_{\varepsilon}$ given by 
\begin{equation}
	\label{Def-J-eps}
	J_\varepsilon(z_T) = \frac{1}{2} \int_0^T \max \{ \| B_1^* z(t)\|_{U_1}^2, \alpha(t) \| B_2^* z(t)\|_{U_2}^2 \} \, dt + \varepsilon \| z_T\|_H +  \langle y_0, z(0) \rangle_H, 
\end{equation}
where $z$ is the solution of the adjoint problem \eqref{Adj-Eq}, and $\alpha = \alpha(t)$ is as in \eqref{Example-Alpha} for a suitable choice of $\omega \in \R^*$.

Details of the proof are left to the reader.

Similarly, counterparts of Theorem \ref{Thm-Main-Gal} and Theorem \ref{Thm-Ext} can also be proved in the context of approximate controllability, by penalizing the functional under consideration by the additional term $\varepsilon \| z_T\|_H$ as in \eqref{Def-J-eps}, the rest of the proof being completely similar. Precise statements and proofs are left to the reader.

\paragraph{Handling source terms.} In the proof of Theorem \ref{Thm-Main}, Theorem \ref{Thm-Main-Gal} and Theorem \ref{Thm-Ext}, we assume that system \eqref{Main-Eq-u} is null-controllable in arbitrary small times. As we said earlier, this is equivalent to say that for all $T >0$, any solution $z$ of \eqref{Adj-Eq} with initial datum $z_T \in H$ satisfies \eqref{Obs-T=0}. It is then easy to check that this property implies that for all $z_T \in H$, the solution $z$ of \eqref{Adj-Eq} satisfies 
\[
	\frac{1}{T} \int_0^T \frac{1}{C_{T-t}^2} \| z(t) \|_H^2 \, dt 
	\leq 
	\sup_{(0,T)}\left\{ \frac{1}{C_{T-t}^2} \| z(t) \|_H^2 \right\} \leq \| B^* z\|_{L^2(0,T; U)}^2, 
\]
thus entailing the existence of a positive function $\rho_T \in L^1_{loc}([0,T))$ such that 
\begin{equation}
	\label{Integrated-Obs}
	\int_0^T \rho_T(t)^2 \|z(t)\|_H^2 \, dt \leq \|B^* z\|_{L^2(0,T; U)}^2. 
\end{equation}
{Besides, easy considerations allow to show that $\rho_T$ can be chosen as a strictly positive function which may degenerate to zero only as $t \to T$. }

This allows to handle source terms in the control problems corresponding to \eqref{Main-Eq-u}. For simplicity, as before, we only focus on the counterpart of Theorem \ref{Thm-Main}, since the counterparts of Theorems \ref{Thm-Main-Gal} and \ref{Thm-Ext} can be done similarly. 

\begin{theorem}
	\label{Thm-Main-source}
	Let us assume $A: \mathscr{D}(A) \subset H \to H$ is a self-adjoint positive definite operator with compact resolvent, $H$ being a Hilbert space.

	Let $B \in \mathscr{L}(U,H)$, where $U$ is an Hilbert space, and assume that $U$ is isomorphic to $U_1 \times U_2$ for some Hilbert spaces $U_1$ and $U_2$, and define $B_1$ and $B_2$ as in \eqref{Def-B1-B2}.

	Assume that system \eqref{Main-Eq-u} is null-controllable in arbitrary small times and satisfies the observability inequality \eqref{Integrated-Obs} for some functions $(\rho_T)_{T>0}$ {a.e.} strictly positive with $\rho_T \in L^1_{loc}(0,T)$.

	Then given any $T>0$, any $y_0 \in H$ and $f \in L^2(0,T; H)$ satisfying
	\begin{equation}
		\label{Integrability-Cond}
		\int_0^T \frac{1}{\rho_T(t)^2} \| f(t) \|_H^2 \, dt < \infty, 
	\end{equation}
	there exist control functions $u_1 \in L^2(0,T; U_1)$ and $u_2 \in L^2(0,T; U_2)$ such that the solution $y$ of 
	\begin{equation}
		\label{Main-Eq-u-1-2-with-f}
		y' + A y = B_1 u_1+ B_2 u_2 + f, \quad {t \in (0,T)},  \qquad y(0) = y_0,  
	\end{equation}
	satisfies \eqref{Null-Cont-Req} while the control functions satisfy the switching condition \eqref{Def-Switching-Controls}.
\end{theorem}

Again, the proof of Theorem \ref{Thm-Main-source} can be easily adapted from the proof of Theorem \ref{Thm-Main} by minimizing, instead of the functional $J$ in \eqref{Def-J-H}, the functional $J_s$ defined for $z_T \in H$ by 
\begin{equation}
	\label{Def-J-source-term}
	J_s(z_T) = \frac{1}{2} \int_0^T \max \{ \| B_1^* z(t)\|_{U_1}^2, \alpha(t) \| B_2^* z(t)\|_{U_2}^2 \} \, dt  + \int_0^T \langle f(t), z(t) \rangle_H \, dt+  \langle y_0, z(0) \rangle_H, 
\end{equation}
where $z$ is the solution of the adjoint problem \eqref{Adj-Eq}, and $\alpha = \alpha(t)$ is as in \eqref{Example-Alpha} for a suitable choice of $\omega \in \R^*$.

The condition \eqref{Integrability-Cond} is there to guarantee that the term 
\[
	  \int_0^T \langle f(t), z(t) \rangle_H \, dt, 
\]
is well-defined in the space $X$ in \eqref{Def-X-Space}, and to preserve the coercivity of the functional $J_s$. Again, the rest of the proof of Theorem \ref{Thm-Main-source} is a verbatim copy of the one of Theorem \ref{Thm-Main} and is left to the reader.

The interest of Theorem \ref{Thm-Main-source} is that it allows to handle source terms and therefore paths the way to prove local null-controllability results with switching controls for semi-linear equations in the presence of superlinear non-linearities.

To do so, one should add suitable weights in the design of the controls. These weights can depend only in time, as in the work \cite{LiuTakahashiTucsnak} based on the knowledge of the cost of controllability in small times, or to more general weights depending on time and space variables as it occurs naturally when using Carleman estimates, see e.g. \cite{FernandezCaraGuerrero-2006, FursikovImanuvilov}.

\subsection{Open problems}\label{Subsec-Op-Pbs}
\paragraph{Time-dependent coefficients.} One of the important restrictions of our approach is that it is based on spectral decompositions of the space, and seems therefore to be strongly limited to operators which are independent of time. It is natural to discuss this property more closely. In fact, looking at our proof, it seems that the only relevant assumption should be an analytic dependence of the operators with respect to the time $t$. However, so far, this problem seems to be out of reach. 

\paragraph{Positive time of controllability.} Our arguments are limited to the case of analytic semigroups which are null-controllable in arbitrary small times, but several results have shown in the last years that there are analytic semigroups which are null-controllable only after some strictly positive critical time. This is the case for instance for the $1$-d heat equation controlled from one well-chosen point, see \cite{Dolecki-1973}, or when considering Grushin operators (see \cite{Beauchard-Darde-Erv-2017} and references therein). 

Our proofs fail to handle these cases, since we do no know how to prove that for $Z_T \in X$ (defined in \eqref{Def-X-Space}), {the function $t \mapsto B^* Z(t)$ (also $t\mapsto B_1^* Z(t)$, $t\mapsto B_2^* Z(t)$) is analytic} in time on strict subintervals of $(0,T)$, which is an essential element of our analysis in the study of the set $I$ in \eqref{Def-I-Switching-Times}.

%

%
%%%%%%%%%%%%%%%%%%%%%%%%%%%%%%%%%%%%%%%%%%%
%
\bibliographystyle{plain}
%\bibliography{/Users/servedoza/Desktop/Biblio} 

%
%
\end{document}